\documentclass[a4paper,10pt,reqno]{amsart}
\usepackage[utf8]{inputenc}
\usepackage{amssymb}
\usepackage{amsthm}
\usepackage{tikz}
\usepackage{mathrsfs}
\usepackage{epsf}
\usepackage{pdfpages}
\usepackage{bookmark,hyperref}
\hypersetup{
     colorlinks=true,
     linkcolor=blue,
     filecolor=blue,
     citecolor = blue,
     urlcolor=cyan,
     }

\usepackage{pstricks}

\newtheorem{theorem}{Theorem}[section]
\newtheorem{lemma}[theorem]{Lemma}
\newtheorem{corollary}[theorem]{Corollary}
\newtheorem{proposition}[theorem]{Proposition}
\newtheorem{question}{Question}
\newtheorem*{definition}{Definition}
\newtheorem*{theorem*}{Theorem}
\newtheorem*{proposition*}{Proposition}
\newtheorem{claim}{Claim}

\newtheoremstyle{mythm}%
{3pt}
{3pt}
{}
{}
{\bfseries}
{.}
{.5em}
{}%

\theoremstyle{mythm}
\newtheorem*{notation}{Notation}
\newtheorem*{example}{Example}
\newtheorem{remark}[theorem]{Remark}

\newcommand{\interior}[1]{{\kern0pt#1}^{\mathrm{o}}}

\def\Z{\mathbb{Z}}
\def\Q{\mathbb{Q}}
\def\R{\mathbb{R}}
\def\C{\mathbb{C}}

\def\set#1{\left\{\, #1 \,\right\}}
\def\abs #1{\left| \,#1\, \right|}
\def\norm #1{\left\| \,#1\, \right\|}
\def\inner #1#2{\langle \,#1,#2\, \rangle}

\def\S{\mathbb{S}}

\def\calA{\mathcal{A}}
\def\calB{\mathcal{B}}
\def\calC{\mathcal{C}}
\def\calD{\mathcal{D}}
\def\calE{\mathcal{E}}

\def\calG{\mathcal{G}}
\def\calH{\mathcal{H}}

\def\calL{\mathcal{L}}

\def\calP{\mathcal{P}}

\def\calS{\mathcal{S}}

\def\mfu{\mathfrak{u}}

\title[Viscosity solutions and hyperbolic motions]
{Viscosity solutions and hyperbolic motions:\\
a new PDE method for the N-body problem}
\author{Ezequiel Maderna}
\address{IMERL \& CMAT, Universidad de la República, Uruguay}
\email{eze@fing.edu.uy}
\author{Andrea Venturelli}
\address{Laboratoire de Mathématiques d'Avignon , France}
\email{andrea.venturelli@univ-avignon.fr}
\thanks{This work was supported by MATH AmSud Sidiham,
CSIC grupo 618 and IFUM LIA-CNRS}
\date{\today}

\begin{document}

\begin{abstract}
We prove for the $N$-body problem
the existence of hyperbolic motions
for any prescribed limit shape and any given initial configuration of the bodies.
The energy level $h>0$ of the motion can also be chosen arbitrarily.
Our approach is based on the construction of
global viscosity solutions for the Hamilton-Jacobi equation $H(x,d_xu)=h$.
We prove that these solutions
are fixed points of the associated Lax-Oleinik semigroup.
The presented results can also be viewed as a new application of
Marchal's Theorem,
whose main use in recent literature
has been to prove the existence of periodic orbits.
\end{abstract}

\maketitle
\setcounter{tocdepth}{3}
\tableofcontents

\section{Introduction}

This paper is about the Newtonian model of gravitation,
also known as the classical $N$-body problem.
We start recalling the standard notation.
Let $E$ be an Euclidean space,
in which the punctual masses $m_1,\dots,m_N>0$
are moving under the action of the inverse-square law of universal gravitation.
If the components of $x=(r_1,\dots,r_N)\in E^N$
are the positions of the bodies,
then we shall denote $r_{ij}=\norm{r_i-r_j}_E$
the distance between bodies $i$ and $j$
for any pair $1\leq i<j\leq N$.
The Newton's equations can be written as $\ddot x=\nabla U(x)$,
where $U:E^N\to\R\cup\set{+\infty}$ is the \emph{Newtonian potential}, 
\[
U(x)=\;\sum_{i<j}\;m_i\,m_j\;r_{ij}^{-1}\;,
\]
and the gradient is taken with respect to the mass scalar product.
A configuration $x\in E^N$ is said to be without collisions if
$U(x)<+\infty$,
that is to say, whenever we have $r_{ij}\neq 0$ for all $i\neq j$.
We denote $\Omega\subset E^N$ the open and dense set of configurations without collisions.
Therefore Newton's equations define an analytic local flow on
$T\Omega=\Omega\times E^N$,
with a first integral given by the energy constant
\[h=\frac{1}{2}\norm{\dot x}^2-U(x)\,.\]

One of the main difficulties
for the analysis of the dynamics in this model is the uncertainty,
for a given motion,
about the presence of singularities after a finite amount of time.
That is to say,
we can not predict whether a certain evolution of the bodies
will be defined for all future time or not.
We recall that maximal solutions that end in finite time
must undergo collisions at the last moment,
or to have an extremely complex behaviour called
pseudocollision (\cite{Dia}, p.39).
Notwithstanding,
the classification of all possible final evolutions was developed,
for motions assumed to be without singularities in the future,
essentially in terms of
the asymptotic behaviour of the distance between the bodies.
Some of the greatest contributions in this direction
are undoubtedly those due to Chazy,
and especially those that he obtained in the works
\cite{Cha1, Cha2} that we comment below.
However,
this approach does not provide the existence of motions
for any type of final evolution.

In this paper we will be concerned with the class of
hyperbolic motions,
as defined by Chazy by analogy with the Keplerian case.

\begin{definition}
Hyperbolic motions are those such that
each body has a different limit velocity vector, that is
$\dot{r}_i(t)\to a_i\in E$ as $t\to +\infty$,
and $a_i\neq a_j$ whenever $i\neq j$. 
\end{definition}

If $V$ is a normed vector space
and $x(t)$ is a smooth curve in $V$
with asymptotic velocity $a\in V$,
then we must have $x(t)=ta+o(t)$ as $t\to +\infty$,
but the converse is of course not true.
However, for $V=E^N$ and $a=(a_1,\dots,a_N)\in\Omega$,
the converse is satisfied
by solutions of the Newtonian $N$-body problem
(see Lemma \ref{lema-cont.limitshape}).
Thus, hyperbolic motions are characterized as
\emph{motions without singularities in the future and such that
 $x(t)=ta+o(t)$ for some configuration $a\in\Omega$.}

It follows that for any hyperbolic motion we have
$\alpha t < r_{ij}(t) <  \beta t$ for some positive constants,
for all $i<j$, and for all $t$ big enough.
As we will see,
Chazy proved that this weaker property
also characterizes hyperbolic motions. 

As usual,
$I(x)=\inner{x}{x}=\sum_i m_i\inner{r_i}{r_i}_E$ will denote the
\emph{moment of inertia} of the configuration $x\in E^N$
with respect to the origin of $E$.
When the motion $x(t)$ is given,
we will use the notation $U(t)$ and $I(t)$
for the compositions $U(x(t))$ and $I(x(t))$ respectively.
Thus for an hyperbolic motion such that
$x(t)=a t + o(t)$ as $t\to +\infty$ we have
$U(t)\to 0$, $I(t)\sim I(a)\,t^2$ and $2h=I(a)>0$.

We say that a motion $x(t)$ has \emph{limit shape}
when there is a time dependent similitude $S(t)$ of the space $E$
such that $S(t)x(t)$ converges to some configuration $a\neq 0$
(here the action of $S(t)$ on $E^N$ is the diagonal one).
Thus the limit shape of an hyperbolic motion is the shape
of his asymptotic velocity $a=\lim_{t\to +\infty}t^{-1}x(t)$.
Note that, in fact,
this represents a stronger way of having a limit shape,
since in this case the similarities are given by homotheties.

\subsection{Existence of hyperbolic motions}
\label{s-existence}

The only explicitly known hyperbolic motions
are of the homographic type,
meaning that the configuration is
all the time in the same similarity class.
For this kind of motion,
$x(t)$ is all the time a central configuration,
that is, a critical point of $I^{1/2}U$.
This is a strong limitation,
for instance the only central configurations for $N=3$
are either equilateral or collinear.
Moreover,
the Painlevé-Wintner conjecture states that up to similarity
there are always a finite number of central configurations.
The conjecture was confirmed by Hampton and Moeckel
\cite{HamMoe} in the case of four bodies,
and by Albouy and Kaloshin \cite{AlbKal}
for generic values of the masses in the planar five-body problem.

On the other hand,
Chazy proved in \cite{Cha2} that the set of initial conditions
giving rise to hyperbolic motions is an open subset of $T\Omega$,
and moreover,
that the limit shape depends continuously on the initial condition
(see Lemma \ref{lema-cont.limitshape}).
In particular,
a motion close enough to some hyperbolic homographic motion
is still hyperbolic.
However, this does not allow us to draw conclusions
about the set of configurations that are realised as limit shapes.
In this paper we prove that \emph{any} configuration
without collisions is the limit shape of some hyperbolic motion.
At our knowledge, there are no results in this direction
in the literature of the subject.

An important novelty in this work is the use of global viscosity solutions, in the sense introduced by Crandall, Evans and Lions \cite{CraLio,CraEvaLio},
for the supercritical Hamilton-Jacobi equation
\begin{equation}\tag{HJ}\label{HJh}
H(x,d_xu)=h \qquad x\in E^N,
\end{equation}
where $H$
is the Hamiltonian of the Newtonian $N$-body problem,
and $h>0$.

We will found global viscosity solutions through a limit process
inspired by the Gromov's construction of the ideal boundary
of a complete locally compact metric space.
To do this,
we will have to generalize to the case $h>0$ the Hölder estimate
for the action potential discovered by the first author in
\cite{Mad1} in the case $h=0$.
With this new estimate we will remedy the loss of
the Lipschitz character of the viscosity subsolutions,
which is due to the singularities of the Newtonian potential.

In a second step, we will show that the functions thus obtained
are in fact fixed points of the Lax-Oleinik semigroup. 
Moreover,
we will prove that given any configuration without collisions
$a\in\Omega$,
there are solutions of Equation (\ref{HJh}) such that
all its calibrating curves are hyperbolic motions
having the shape of $a$ as limit shape.
Following this method (developed in Sect. \ref{s-HJ})
we get to our main result.

\begin{theorem}
\label{thm-princ}
For the Newtonian $N$-body problem in a space $E$
of dimension at least two,
there are hyperbolic motions $x:[0,+\infty)\to E^N$ such that
\[x(t)=\sqrt{2h}\,t\;a+o(t)\quad\text{ as }\quad t\to +\infty,\]
for any choice of $x_0=x(0)\in E^N$,
for any configuration without collisions $a\in\Omega$ normalized by $\norm{a}=1$,
and for any choice of the energy constant $h>0$.  
\end{theorem}

We emphasize the fact that the initial configuration can be chosen \emph{with} collisions.
This means that in such a case, the motion $x$ given by the theorem is continuous at $t=0$,
and defines a maximal solution $x(t)\in\Omega$ for $t>0$. 
For instance,
choosing $x_0=0\in E^N$,
the theorem gives the existence of ejections from the total collision configuration,
with prescribed positive energy and arbitrarily chosen limit shape.

Moreover, the well known Sundman's inequality (see Wintner \cite{Win}) implies that motions with total collisions have zero angular momentum. 
Therefore, we deduce the following non trivial corollary.

\begin{corollary}
For any configuration without collisions $a\in\Omega$
there is a hyperbolic motion with zero angular momentum
and asymptotic velocity $a$.  
\end{corollary}

It should be said that the hypothesis that excludes
the collinear case $\dim E=1$ is only required to ensure
that action minimizing curves do not suffer collisions.
The avoidance of collisions is thus assured by the celebrated
Marchal's Theorem that we state below in Sect. \ref{s-var.setting}.
The collinear case could eventually be analyzed
in the light of the results obtained by Yu and Zhang \cite{YuZha}.

Theorem \ref{thm-princ} should be compared with that obtained
by the authors in \cite{MadVen} which concerns completely parabolic motions.
We recall that completely parabolic motions (as well as total collisions)
have a very special asymptotic behaviour.
In his work of 1918 \cite{Cha1},
Chazy proves that the normalized configuration
must approximate the set of normal central configurations.
Under a hypothesis of non-degeneracy,
he also deduces the convergence to a particular central configuration. 
This hypothesis is always satisfied in the three body problem.
However,
a first counterexample with four bodies in the plane
was founded by Palmore \cite{Pal},
allowing thus the possibility of motions
with infinite spin (see Chenciner \cite{Che1} p.281).

In all the cases, Chazy's Theorem prevents arbitrary limit shapes
for completely parabolic motions as well as for total collisions.
In this sense,
let us mention for instance the general result by Shub \cite{Shu}
on the localisation of central configurations,
showing that they are isolated from the diagonals.

We should also mention that the confinement of the
asymptotic configuration to the set of central configurations,
both for completely parabolic motions and for total collisions,
extends to homogeneous potentials of degree $\alpha\in (-2,0)$.
For these potentials the mutual distances must grow like $t^{2/(2-\alpha)}$ in
the parabolic case, and must decay like $\abs{t-t_0}^{2/(2-\alpha)}$
in the case of a total collision at time $t=t_0$.
On the other hand,
it is known that potentials giving rise to strong forces near collisions
can present motions of total collision with non-central asymptotic configurations.
We refer the reader to the comments on the subject by Chenciner in \cite{Che3}
about the Jacobi-Banachiewitz potential, and to Arredondo et al. \cite{ArPeChSt}
for related results on the dynamics of total collisions in the case of
Schwarzschild and Manev potentials.

Let us say that there is another natural way to prove the existence of hyperbolic motions, 
using the fact that the Newtonian force vanishes when all mutual distances diverge.
We could call these motions \emph{almost linear}.
The way to do that is as follows.
Suppose first that $(x_0,a)\in \Omega\times\Omega$
is such that the half-straight line given by $\bar{x}(t)=x_0+ta$, $t>0$
has no collisions ($\bar{x}(t)\in\Omega$ for all $t>0$).
Consider now the motion $x(t)$ with initial condition
$x(0)=x_0$ and $\dot x(0)=\alpha a$ for some positive constant $\alpha$.
It is not difficult to prove that, for $\alpha>0$ chosen big enough,
the trajectory $x(t)$ is defined for all $t>0$, and moreover,
it is a hyperbolic motion with limit velocity $b\in\Omega$ close to $\alpha a$.
In particular,
the limit shape of such a motion can be obtained
as close as we want from the shape of $a$.

The previous construction is unsatisfactory for several reasons.
First, we do not get exactly the desired limit shape but a close one.
This approximation can be made as good as we want,
but we lose the control of the energy constant $h$ of the motion,
whose order of magnitude is that of $\alpha^2$.
Secondly,
it is not possible to apply this method when the half-straight line
$\bar{x}$ presents collisions.
For instance this is the case if we take $a=z_0-x_0$
for any choice of $z_0\in E^N\setminus\Omega$.
Finally,
even if the homogeneity of the potential can be exploited
to find a new hyperbolic motion
with a prescribed positive energy constant,
and the same limit shape,
we lose the control on the initial configuration.
Indeed,
if $x$ is a hyperbolic motion defined for all $t\geq 0$
with energy constant $h$,
then the motion $x_\lambda$ defined by
$x_\lambda(t)=\lambda\,x(\lambda^{-3/2}t)$
is still hyperbolic with energy constant $\lambda^{-1}h$.
Moreover, the limit shapes of $x$ and $x_\lambda$ are the same,
but $x_\lambda(0)=\lambda x(0)$ meaning that
the initial configuration is dilated by the factor $\lambda$.


\subsection{Other expansive motions}
\label{s-other}

Hyperbolic motions are part of the family of
\emph{expansive motions} which we define now.
In order to classify them,
as well as for further later uses,
we summarize below a set of well-known facts
about the possible evolutions of the motions
in the Newtonian $N$-body problem.

\begin{definition}
[Expansive motion]
A motion $x:[0,+\infty)\to \Omega$ is said to be expansive
when all the mutual distances diverge.
That is, when $r_{ij}(t)\to+\infty$ for all $i<j$.
Equivalently, the motion is expansive if $U(t)\to 0$.
\end{definition}

We will see that there are three well defined classes of expansive motions. 
First of all we must observe that,
since we $U(t)\to 0$ implies $\norm{\dot x(t)}\to\sqrt{2h}$,
expansive motions can only occur with $h\geq 0$.

In his pioneering work,
Jean Chazy proposed a classification of motions in terms of their final evolution.
In the Keplerian case there is only one distance function to consider,
and the three classes of motions are elliptic, parabolic and hyperbolic.
Extending the analysis for $N\geq 3$,
he introduced several hybrid classes of motions,
such as hyperbolic-elliptical in which some distances diverge and others remain bounded.
In his attempt to achieve a full classification,
he obtains the theoretical possibility of complex behaviours
such as the so-called oscillatory motions or the superhyperbolic motions,
see Saari and Xia \cite{SaaXia}.
After the works of Chazy, and for quite some time,
specialists have doubted the existence of such motions
because of his complex and paradoxical appearance. 
The same can be said about the existence of pseudo-collision
singularities, which, as is well known, are impossible if $N=3$.

Let us say that the existence results of oscillatory motions goes
back to the work of Sitnikov \cite{Sit} for the spatial restricted
three-body problem.
Then, the main idea in this paper was extended to
the unrestricted problem by Alekseev
(see Moser \cite{Mos} for a more detailed explanation of this
and other related developments).
Sitnikov's ideas were undoubtedly very important
for the construction of the first example of a motion
with a pseudo-collision singularity with five bodies by Xia \cite{Xia}.
With respect to superhyperbolic motions we must say that,
although there are no known examples of them,
they exist at least in a weak sense
for the collinear four-body problem
(with regularisation of binary collisions) \cite{SaaXia}.

As we will see, to achieve the proof of the announced results,
it will be crucial to show that certain motions that will be obtained
are not superhyperbolic,
and that they do not suffer collisions nor pseudo-collisions.

We need to introduce two functions which play an important role
in the classical description of the dynamics.
For a given motion, these two functions are
\[
r(t)=\min_{i<j}r_{ij}(t)\quad \text{ and }\quad
R(t)=\max_{i<j}r_{ij}(t)\,
\]
the minimum and the maximum separation between the bodies
at time $t$.
We now recall some  facts concerning the possible behaviours
of the trajectories as $t\to +\infty$ in terms of the behaviours
of these functions.

We start by fixing some notation and making some remarks.

\begin{notation}
Given positive functions $f$ and $g$,
we will write $f\approx g$ when the quotient of them
is bounded between two positive constants.
\end{notation}

\begin{remark}
It is easy to see that $r\approx U^{-1}$.
Moreover, 
$R^2\approx I_G$ where $I_G$  denotes the moment of inertia
with respect to the center of mass $G$ of the configuration.
To see this it suffices to write $I_G$
in terms of the mutual distances.
\end{remark}

\begin{remark}
\label{rmk-configur.measure}
The function $\mu=U\,I_G\,^{1/2}$ is homogeneous
of degree zero.
Some authors call this function the
\emph{configurational measure}.
According to the previous remark we have $\mu\approx R\,r^{-1}$.
\end{remark}

\begin{remark}
By König's decomposition we have that $I=I_G+M\norm{G}_E^2$
where $M$ is the total mass of the system.
Therefore, using the Largange-Jacobi identity $\ddot I=4h+2U$ we deduce that,
if $h>0$ and the center of mass is at rest,
then $R(t)>At$ for some constant $A>0$.
\end{remark}

\begin{theorem*}[1922, Chazy \cite{Cha2} pp. 39 -- 49]
Let $x(t)$ be a motion with energy constant $h>0$ and defined for all $t>t_0$.

\begin{enumerate}
\item[(i)] The limit
\[\lim_{t\to +\infty}R(t)\,r(t)^{-1}=L\in [1,+\infty]\]
always exists.

\item[(ii)] If $L<+\infty$ then there is a configuration $a\in\Omega$,
and some function $P$, which is analytic in a neighbourhood of $(0,0)$,
such that for every $t$ large enough we have
\[x(t)=ta-\log(t)\,\nabla U(a)+P(u,v)\]
where $u=1/t$ and $v=\log(t)/t$.
\end{enumerate}
\end{theorem*}

As Chazy pointed out, surprisingly Poincaré made the mistake of 
omitting the $\log(t)$ order term in his
\emph{``Méthodes Nouvelles de la Mécanique Céleste"}.

Subsequent advances in this subject were recorded much later,
when Chazy's results on final evolutions were included
in a more general description of motions.
From this development we must recall the following theorems.
Notice that none of them make assumptions
on the sign of the energy constant $h$.

\begin{theorem*}[1967, Pollard \cite{Pol}]
Let $x$ be a motion defined for all $t>t_0$.
If $r$ is bounded away from zero then we have that
$R=O(t)$ as $t\to +\infty$.
In addition $R(t)/t\to +\infty$ if and only if $r(t)\to 0$.
\end{theorem*}

This leads to the following definition.

\begin{definition} A motion is said to be superhyperbolic when
\[\limsup_{t\to +\infty}\;R(t)/t=+\infty.\]
\end{definition}

A short time later it was proven that,
either the quotient $R(t)/t\to +\infty$,
or $R=O(t)$ and the system expansion can be described more accurately.

\begin{theorem*}[1976, Marchal-Saari \cite{MarSaa}]
Let $x$ be a motion defined for all $t>t_0$.
Then either $R(t)/t \to +\infty$ and $r(t)\to 0$,
or there is a configuration $a\in E^N$ such that $x(t)=ta+O(t^{2/3})$.
In particular, for superhyperbolic motions the quotient $R(t)/t$ diverges.
\end{theorem*}

Of course this theorem does not provide much information
in some cases, for instance if the motion is bounded then
we must have $a=0$.
On the other hand,
it admits an interesting refinement concerning the
the behaviour of the subsystems.
More precisely,
when $R(t)=O(t)$ and the configuration $a$ given by the theorem
has collisions the system decomposes naturally into subsystems,
within which the distances between the bodies
grow at most like $t^{2/3}$.
Considering the internal energy of each subsystem,
Marchal and Saari (Ibid, Theorem 2 and corollary 4 pp.165-166)
gave a decription of the possible dynamics
that can occur within the subsystems.
From these results we can easily deduce the following.

\begin{theorem*}[1976, Marchal-Saari \cite{MarSaa}]
Suppose that $x(t)=ta+O(t^{2/3})$ for some $a\in E^N$,
and that the motion is expansive.
Then, for each pair $i<j$ such that $a_i=a_j$ we have $r_{ij}\approx t^{2/3}$.
\end{theorem*}

Notice that we can always consider the \emph{internal motion}
of the system, that is, looking at the relative positions
of the bodies with respect to their center of mass.
This gives a new motion with the same distance functions.
Moreover,
the internal motion of an expansive motion is also expansive.

All the previous considerations allow us to classify expansive motions
according to the asymptotic order of growth of the distances between the bodies.
Since an expansive motion is not superhyperbolic,
we can assume that it is of the form $x(t)=ta+O(t^{2/3})$ for some $a\in E^N$.
Moreover, we can assume that the center of mass is at rest,
meaning that $G(a)=0$.
We get then the following three types.

\begin{enumerate}
\item[(H) ] \emph{Hyperbolic} :
$a\in\Omega$, and $r_{ij}\approx t$ for all $i<j$
\item[(PH)] \emph{Partially hyperbolic} :
$a\in E^N\setminus\Omega$ but $a\neq 0$.
\item[(P) ] \emph{Completely parabolic} :
$a=0$, and $r_{ij}\approx t^{2/3}$ for all $i<j$.
\end{enumerate}

Let $h_0$ be the energy constant of the above defined
internal motion.
It is clear that the first two types can only occur if $h_0>0$,
while the third requires $h_0=0$.

Finally,
we observe that Chazy's Theorem applies in the first two cases .
In these cases,
the limit shape of $x(t)$ is the shape of the configuration $a$
and moreover,
we have $L<+\infty$ if and only if $x$ is hyperbolic.
Of course if $h_0>0$ and $L=+\infty$ then
either the motion is partially hyperbolic or it is not expansive.


\subsection{The geometric viewpoint}
\label{s-geom.view}

We explain now the geometric formulation
and the geometrical meaning of this work
with respect to the Jacobi-Maupertuis metrics associated
to the positive energy levels.
Several technical details concerning these metrics are given
in Sect. \ref{s-jm.dist}.
The boundary notions are also discussed in Sect. \ref{s-busemann}.
It may be useful for the reader to keep in mind
that reading this section can be postponed to the end.

We recall that for each $h\geq 0$, the Jacobi-Maupertuis metric
of level $h$ is a Riemannian metric defined
on the open set of configurations without collisions $\Omega$.
More precisely,
it is the metric defined by $j_h=2(h+U)\,g_m$,
where $g_m$ is the Euclidean metric in $E^N$
given by the mass inner product.
Our main theorem has a stronger version
in geometric terms.
Actually Theorem \ref{thm-princ} can be reformulated
in the following way.

\begin{theorem}\label{thm-princ.rays}
For any $h>0$, $p\in E^N$ and $a\in\Omega$,
there is  geodesic ray of the Jacobi-Maupertuis metric of level $h$
with asymptotic direction $a$ and starting at $p$.
\end{theorem}
This formulation requires some explanations.
The Riemannian distance $d_h$ in $\Omega$ is defined as usual
as the infimum of the length functional among all the
piecewise $C^1$ curves in $\Omega$ joining two given points.
We will prove that $d_h$ can be extended
to a distance $\phi_h$ in $E^N$,
which is a metric completion of $(\Omega,d_h)$,
and which also we call the Jacobi-Maupertuis distance.
Moreover, we will prove that $\phi_h$ is precisely the action
potential defined in Sect. \ref{s-var.setting}.

The minimizing geodesic curves can then be defined
as the isometric immersions of compact intervals
$[a,b]\subset\R$ within $E^N$.
Moreover, we will say that a curve $\gamma:[0,+\infty)\to E^N$
is a geodesic ray
from $p\in E^N$, if $\gamma(0)=p$ and each restriction
to a compact interval is a minimizing geodesic.
To deduce this geometric version of our main theorem
it will be enough to observe that the obtained
hyperbolic motions can be reparametrized
taking the action as parameter to obtain geodesic rays.

We observe now the following interesting implication of
Chazy's Theorem.
\begin{remark}
\label{rmk-Chazy.implic}
If two given hyperbolic motions have the same asymptotic
direction, then they have a bounded difference.
Indeed,
if $x$ and $y$ are hyperbolic motions with the same asymptotic
direction, then the two unbounded terms
of the Chazy's asymptotic development of $x$ and $y$ also agree.
\end{remark}

We recall that the Gromov boundary of a geodesic space
is defined as the quotient set of the set of geodesic rays
by the equivalence that identifies rays that are kept
at bounded distance.
From the previous remark,
we can deduce that two geodesic rays with asymptotic direction
given by the same configuration $a\in\Omega$
define the same point at the Gromov boundary.

\begin{notation}
Let $\phi_h:E^N\times E^N\to\R^+$ be the Jacobi-Maupertuis
distance for the energy level $h\geq 0$
in the full space of configurations.
We will write $\calG_h$ for the corresponding Gromov boundary.
\end{notation}

The proof of the following corollary is given in Sect. \ref{s-jm.dist}.

\begin{corollary}\label{cor-Gr.boundary}
If $h>0$, then each class in $\Omega_1=\Omega/\R^+$
determines a point in $\calG_h$ which is composed
by all geodesic rays with asymptotic direction in this class.
\end{corollary}

On the other hand,
if instead of the arc length we parametrize the geodesics
by the dynamical parameter,
then it is natural to question the existence
of non-hyperbolic geodesic rays.
We do not know if there are partially hyperbolic geodesic rays. 
Nor do we know if a geodesic ray should be an expansive motion. 

In what follows we denote $\norm{v}_h$
for the norm  of $v\in T\Omega$ with respect to the metric $j_h$,
and $\norm{p}_h$ the dual norm of an element $p\in T^*\Omega$.
If $\gamma:(a,b)\to\Omega$ is a geodesic
parametrized by the arc length, then
\[\norm{\dot\gamma(s)}_h^2=
2(h+U(\gamma(s))\,\norm{\dot\gamma(s)}^2=1\]
for all $s\in (a,b)$.
Taking into account that $U\approx r^{-1}$ we see that
the parametrization of motions as geodesics
leads to slowed evolutions over passages near collisions.
We also note that for expansive geodesics we have
$\norm{\dot\gamma(s)}\to 1/\sqrt{2h}$.

Finally we make the following observations about the
Hamilton-Jacobi equation that we will solve in the weak sense.
First, the equation (\ref{HJh}) , that explicitly reads
\[\frac{1}{2}\norm{d_xu}^2-U(x)=h\]
can be written in geometric terms, precisely as the eikonal equation
\[
\norm{d_xu}_h=\frac{1}{\sqrt{2(h+U(x))}}\,\norm{d_xu}=1
\]
for the Jacobi-Maupertuis metric.
On the other hand,
the solutions will be obtained as limits of weak subsolutions,
which  can be viewed as $1$-Lipschitz functions for the
Jacobi-Maupertuis distance.
We will see that the set of viscosity subsolutions
is the set of functions
$u:E^N\to\R$ such that $u(x)-u(y)\leq \phi_h(x,y)$
for all $x,y\in E^N$.


\section{Viscosity solutions of the Hamilton-Jacobi equation}
\label{s-HJ}

The \emph{Hamiltonian} $H$ is defined over
$T^*E^N\simeq E^N\times (E^*)^N$ as usual by
\[
H(x,p)=\frac{1}{2}\norm{p}^2-U(x)
\]
and taking the value $H(x,p)=-\infty$
whenever the configuration $x$ has collisions.
Here the norm is the dual norm with respect to the mass product,
 that is,
for $p=(p_1,\dots,p_N)\in (E^*)^N$
\[
\norm{p}^2=
m_1^{-1}\,\norm{p_1}^2+\dots+m_N^{-1}\,\norm{p_N}^2\,,
\]
thus in terms of the positions of the bodies
equation (\ref{HJh}) becomes
\[H(x,d_xu)=
\sum_{i=1}^N\frac{1}{2\,m_i}\norm{\frac{\partial u}{\partial r_i}}^2-
\;\sum_{i<j}\frac{m_im_j}{r_{ij}}\; = h\,.\]
As is known,
the method of characteristics for this type of equations
consists in reducing the problem to the resolution
of an ordinary differential equation,
whose solutions are precisely the characteristic curves.
Once these curves are determined,
we can obtain solutions by integration along these curves,
from a cross section in which the solution value is given.
Of course,
here the characteristics are precisely the solutions
of the $N$-body problem and can not be computed.
Our method will be the other way around:
first we build a solution as a limit of subsolutions,
and then we find characteristic curves
associated with that solution.

We will start by recalling
the notion of viscosity solution in our context.
There is an extremely wide literature on viscosity solutions due to
 the great diversity of situations in which they can be applied.
For a general and introductory presentation the books of Evans
 \cite{Eva1} and Barles \cite{Bar} are recommended.
For a broad view on the Lax-Oleinik semigroups
we suggest references \cite{Ber,Eva2, Fat}.  

\begin{definition}[Viscosity solutions]
With respect to the Hamilton-Jacobi equation (\ref{HJh}),
we say that a continuous function $u:E^N\to\R$ is a
\begin{enumerate}
\item \emph{viscosity \bf{subsolution}},
if for any $\psi\in C^1(E^N)$ and for any configuration $x_0$
at which $u-\psi$ has a local maximum we have
$H(x_0,d_{x_0}\psi)\leq h$.

\item \emph{viscosity \bf{supersolution}},
if for any $\psi\in C^1(E^N)$ and for any configuration $x_0$
at which $u-\psi$ has a local minimum we have
$H(x_0,d_{x_0}\psi)\geq h$.

\item \emph{viscosity \bf{solution}},
as long as is both a subsolution and a supersolution.
\end{enumerate}
\end{definition}

\begin{remark}
It is clear that we get the same notions by taking test functions
$\psi$ defined on open subsets of $E^N$.
\end{remark}

\begin{remark}
The notion of viscosity solution is a generalization
of the notion of classical solution.
Indeed, if $u\in C^1(E^N)$
satisfies the Hamilton-Jacobi equation everywhere,
then $u$ is a viscosity solution
since we can take $\psi=u$ as a test function.
\end{remark}

If $u\in C^0(E^N)$ is a viscosity solution,
then we have $H(x,d_xu)=h$
at any point where $u$ is differentiable.
This follows from the fact that for any $C^0$ function $u$,
the differentiability at $x_0\in E^N$
implies the existence of $C^1$ functions
$\psi^-$ and $\psi^+$ such that $\psi^-\leq u \leq \psi^+$ and
$\psi^-(x_0)=u(x_0)=\psi^+(x_0)$.
As we will see (Lemma \ref{lema-visc.subsol.locLip}),
in our case viscosity subsolutions are locally Lipschitz
over the open and dense set $\Omega\subset E^N$
of configurations without collisions.
Therefore by Rademacher's Theorem
they are differentiable almost everywhere.
But, as is well known to the reader familiar with the subject,
being a viscosity solution is a much more demanding property
than satisfying the equation almost everywhere.

\begin{remark}
We note that the participation of the unknown $u$
in equation (\ref{HJh}) is only through the derivatives $d_xu$.
Therefore the set of classical solutions
is preserved by addition of constants.
Also note that the same applies for the set of
viscosity subsolutions and the set of viscosity supersolutions.
\end{remark}

From now on,
we will use of the powerful interaction between
the Hamiltonian view of dynamics and the Lagrangian view.
The Hamilton-Jacobi equation provides a great bridge
between the symplectic aspects of dynamics
and the variational properties of trajectories.

Once the Lagrangian action is defined,
we will characterize the set of viscosity subsolutions
as the set of functions satisfying a property of domination
with respect to the action.
Then, the next step will be to prove the equicontinuity
of the family of viscosity subsolutions
by finding an estimate for an action potential.

\subsection{Action potentials and viscosity solutions} 
\label{s-var.setting}

The \emph{Lagrangian} is defined on
$TE^N\simeq E^N\times E^N$ fiberwise
as the convex dual of the Hamiltonian, that is
\[
L(x,v)=\sup\set{p(v)-H(x,p)\mid p\in (E^*)^N}
\]
or equivalently,
\[
L(x,v)=\frac{1}{2}\norm{v}^2 + U(x)\,,
\]
so in particular it takes the value $L(x,v)=+\infty$
if $x$ has collisions.
The Lagrangian action will be considered
on absolutely continuous curves,
and its value could be infinite.
We will use the following notation.
For $x,y\in E^N$ and $\tau>0$, let
\[
\calC(x,y,\tau)=
\set{\gamma:[a,b]\to E^N \mid\gamma(a)=
x,\,\gamma(b)=y,\,b-a=\tau}
\]
be the set of absolutely continuous curves
going from $x$ to $y$ in time $\tau$,
and
\[
\calC(x,y)=\bigcup_{\tau>0}\calC(x,y,\tau)\,.
\]
The Lagrangian action of a curve $\gamma\in\calC(x,y,\tau)$
will be denoted
\[
\calA_L(\gamma)=\int_a^b L(\gamma,\dot\gamma)\,dt=
\int_a^b \tfrac{1}{2}\norm{\dot\gamma}^2+U(\gamma)\,dt
\]

It is well known that Tonelli's Theorem
on the lower semicontinuity of the action for convex Lagrangians
extends to this setting.
A proof can for instance be found in \cite{daLMad} (Theorem 2.3).
In particular we have, for any pair of configurations $x,y\in E^N$,
the existence of curves achieving the minimum value
\[
\phi(x,y,\tau)=
\min\set{\calA_L(\gamma)\mid \gamma\in\calC(x,y,\tau)}
\]
for any $\tau>0$.
When $x\neq y$ there also are curves reaching the minimum
\[
\phi(x,y)=\min\set{\calA_L(\gamma)\mid \gamma\in\calC(x,y)}
=\min\set{\phi(x,y,\tau)\mid \tau>0}
\]
In the case $x=y$ we have
$\phi(x,x)=\inf\set{\calA_L(\gamma)\mid \gamma\in\calC(x,y)}=0$.
We call these functions on $E^N\times E^N$ respectively
the \emph{fixed time action potential} and the
\emph{free time} (or \emph{critical}) \emph{action potential}.

According to the Hamilton's principle of least action,
if a curve $\gamma:[a,b]\to E^N$ is a minimizer of the Lagrangian
action in $\calC(x,y,\tau)$ then $\gamma$ satisfy Newton's
equations at every time $t\in [a,b]$ in which $\gamma(t)$
has no collisions, i.e. whenever $\gamma(t)\in\Omega$.

On the other hand,
it is easy to see that there are curves
both with isolated collisions and finite action.  
This phenomenon,
already noticed by Poincaré in \cite{Poi}, prevented the use
of the direct method of the calculus of variations
in the $N$-body problem for a long time. 

A big breakthrough came from Marchal,
who gave the main idea needed to prove the following theorem.
Complete proofs of this and more general versions were
established by Chenciner \cite{Che1}
and  by Ferrario and Terracini \cite{FerTer}.

\begin{theorem*}[2002, Marchal \cite{Mar}]
If $\gamma\in\calC(x,y)$ is defined on some interval $[a,b]$,
and satisfies $\calA_L(\gamma)=\phi(x,y,b-a)$,
then $\gamma(t)\in\Omega$ for all $t\in (a,b)$.
\end{theorem*}
Thanks to this advance,
the existence of countless periodic orbits has been established
using variational methods.
Among them, the celebrated three-body figure eight
due to Chenciner and Montgomery \cite{CheMon}
is undoubtedly the most representative example,
although it was discovered somewhat before.
Marchal's Theorem was also used to prove the nonexistence
of entire free time minimizers \cite{daLMad},
or in geometric terms,
that the zero energy level has no straight lines.
The proof we provide below for our main result
depends crucially on Marchal's Theorem.
Our results can thus be considered as a new application
of Marchal's Theorem, this time in positive energy levels.

We must also define for $h>0$ the
\emph{supercritical action potential} as the function 
\[
\phi_h(x,y)=
\inf\set{\calA_{L+h}(\gamma)\mid \gamma\in\calC(x,y)}=
\inf\set{\phi(x,y,\tau)+h\tau\mid \tau>0}.
\]

For the reader familiar with the Aubry-Mather theory,
this definition should be reminiscent of the supercritical
action potentials used by Mañé to define the critical value
of a Tonelli Lagrangian on a compact manifold. 

As before we prove
(see Lemma \ref{lema-JM.geod.complet} below),
now for $h>0$,
that given any pair of different configurations $x,y\in E^N$,
the infimum in the definition of $\phi_h(x,y)$
is achieved by some curve $\gamma\in\calC(x,y)$,
that is, we have $\phi_h(x,y)=\calA_{L+h}(\gamma)$.
It follows that if $\gamma$ is defined in $[0,\tau]$,
then $\gamma$ also minimizes $\calA_L$ in $\calC(x,y,\tau)$
and by Marchal's Theorem we conclude that $\gamma$
avoid collisions,
i.e. $\gamma(t)\in\Omega$ for every $t\in (0,\tau)$.

\subsubsection{Dominated functions and viscosity subsolutions}

Let us fix $h>0$
and take a $C^1$ subsolution $u$ of $H(x,d_xu)=h$,
that is, such that $H(x,d_xu)\leq h$ for all $x\in E^N$.
Notice now that,
since for any absolutely continuous curve
$\gamma:[a,b]\to E^N$ we have
\[
u(\gamma(b))-u(\gamma(a))=
\int_a^b d_{\gamma}u(\dot\gamma)\,dt\,,
\]
by Fenchel's inequality we also have
\[
u(\gamma(b))-u(\gamma(a))\leq
\int_a^b L(\gamma,\dot\gamma)+
H(\gamma,d_{\gamma}u)\;dt\;\leq
\;\calA_{L+h}(\gamma)\,.
\]
Therefore we can say that if $u$ is a $C^1$ subsolution,
then
\[
u(x)-u(y)\leq
\calA_{L+h}(\gamma)
\]
for any curve
$\gamma\in\calC(x,y)$.
This motivates the following definition.

\begin{definition}[Dominated functions]
We said that $u\in C^0(E^N)$ is dominated by $L+h$,
and we will denote it by $u\prec L+h$,
if we have
\[
u(y)-u(x)\leq
\phi_h(x,y)\quad\text{ for all }\quad x,y\in E^N.
\]
\end{definition}

Thus we know that $C^1$ subsolutions are dominated functions.
We prove now the well-known fact that dominated functions
are indeed viscosity subsolutions.

\begin{proposition}
\label{prop-dom.are.visc.ssol}
If $u\prec L+h$ then $u$ is a viscosity subsolution of (\ref{HJh}).
\end{proposition}

\begin{proof} 
Let $u\prec L+h$ and consider a test function $\psi\in C^1(E^N)$.
Assume that $u-\psi$ has a local maximum
at some configuration $x_0\in E^N$.
Therefore,
for all $x\in E^N$ we have $\psi(x_0)-\psi(x) \leq u(x_0)-u(x)$.

On the other hand,
the convexity and superlinearity of the Lagrangian
implies that there is a unique $v\in E^N$ such that
$H(x_0,d_{x_0}\psi)=d_{x_0}\psi (v)-L(x_0,v)$.
Taking any smooth curve $x:(-\delta,0]\to E^N$
such that $x(0)=x_0$ and
$\dot x(0)=v$ we can write, for $t\in (-\delta,0)$
\[\frac{\psi(x_0)-\psi(x(t))}{-t}\leq \frac{u(x_0)-u(x(t))}{-t}\leq
\;-\frac{1}{t}\;\calA_{L+h}\left(x\mid_{[t,0\,]}\right)\]
thus for $t\to 0^-$ we get
$d_{x_0}\psi (v)\leq L(x_0,v)+h$,
that is to say, $H(x_0,d_{x_0}\psi)\leq h$ as we had to prove.
\end{proof}
Actually, the converse can be proved.
For all that follows,
we will only need to consider dominated functions,
and for this reason,
it will no longer be necessary to manipulate test functions
to verify the subsolution condition in the viscosity sense.
However,
for the sake of completeness we give a proof of this converse.

A first step is to prove that viscosity subsolutions
are locally Lipschitz on the open, dense,
and full measure set of configurations without collisions
(for this we follow the book of
Bardi and Capuzzo-Dolcetta \cite{BarCap}, Proposition 4.1, p. 62).

\begin{lemma}
\label{lema-visc.subsol.locLip}
The viscosity subsolutions of (\ref{HJh})
are locally Lipschitz on $\Omega$.
\end{lemma}

\begin{proof}
Let $u\in C^0(E^N)$ be a viscosity subsolution and
let $z\in\Omega$.
We take a compact neighbourhood $W$ of $z$
in which the Newtonian potential is bounded,
i.e. such that $W\subset\Omega$.
Thus our Hamiltonian is coercive on $T^*W$,
meaning that given $h>0$ we can choose $\rho>0$ for which,
if $\norm{p}>\rho$ and $w\in W$ then $H(w,p)>h$.

We choose now $r>0$ such that
the open ball $B(z,3r)$ is contained in $W$.
Let $M=\max\set{u(x)-u(y)\mid\;x,y\in W}$
and take $k>0$ such that $2kr>M$.

We take now any configuration $y\in B(z,r)$ and we define,
in the closed ball $\overline{B}_y=\overline{B}(y,2r)$,
the function $\psi_y(x)=u(y)+k\norm{x-y}$.
We will use the function $\psi_y$ as a test function in the open set
$B^*_y=B(y,2r)\setminus\set{y}$.
We observe first that $u(y)-\psi_y(y)=0$ and that
$u-\psi_y$ is negative in the boundary of $\overline{B}_y$.
Therefore the maximum of $u-\psi_y$
is achieved at some interior point
$x_0\in B(y,2r)$.

Suppose that $x_0\neq y$.
Since $\psi_y$ is smooth on $B^*_y$,
and $u$ is a viscosity subsolution,
we must have $H(x_0,d_{x_0}\psi_y)\leq h$.
Therefore we must also have $k=\norm{d_{x_0}\psi_y}\leq\rho$.

We conclude that, if we choose $k>\rho$ such that $2rk>M$,
then for any $y\in B(z,r)$
the maximum of $u-\psi$ in $\overline{B}_y$
is achieved at $y$, meaning that $u(x)-u(y)\leq k\norm{x-y}$
for all $x\in\overline{B}_y$.
This proves that $u$ is $k$-Lipschitz on $B(z,r)$.
\end{proof}

\begin{remark}
\label{rmk-vssol.are.subsol.ae}
By Rademacher's Theorem, we know that any viscosity subsolution
is differentiable almost everywhere in the open set $\Omega$.
In addition, at every point of differentiability $x\in\Omega$
we have $H(x,d_xu)\leq h$.
Therefore, since $\Omega$ has full measure in $E^N$,
we can say that viscosity subsolutions satisfies $H(x,d_xu)\leq h$
almost everywhere in $E^N$.
\end{remark}

\begin{remark}
We observe that the local Lipschitz constant $k$ we have obtained
in the proof depends, a priori,
on the chosen viscosity subsolution $u$.
We will see that this is not really the case.
This fact will result immediatly from the following proposition and
Theorem \ref{thm-phih.estim}.
\end{remark}

We can prove now that
the set of viscosity subsolutions of $H(x,d_xu)=h$
and the set of dominated functions $u\prec L+h$ coincide.

\begin{proposition}
\label{prop-visc.ssol.are.dom}
If $u$ is a viscosity subsolution of (\ref{HJh}) then $u\prec L+h$.
\end{proposition}

\begin{proof}
Let $u:E^N\to\R$ be a viscosity subsolution.
We have to prove that
\[u(y)-u(x)\leq \calA_{L+h}(\gamma)\quad
\text{ for all }x,y\in E^N,\;\gamma\in\calC(x,y)\,.\]
We start by showing the inequality for any segment
$s(t)=x+t(y-x)$, $t\in [0,1]$.
Note first that in the case $y=x$ there is nothing to prove,
since the action is always positive.
Thus we can assume that $r=\norm{y-x}>0$.

We know $H(x,d_xu)\leq h$ is satisfied
on a full measure set $\calD\subset E^N$
in which $u$ is differentiable,
see Lemma \ref{lema-visc.subsol.locLip}
and Remark \ref{rmk-vssol.are.subsol.ae}.
Assuming that $s(t)\in\calD$ for almost every $t\in[0,1]$
we can write
\[
\frac{d}{dt}\,u(s(t))=d_{s(t)}u\,(y-x)\quad\text{ a.e. in }[0,1]
\]
from which we deduce,
applying Fenchel's inequality and integrating,
\[
u(y)-u(x)\leq
\int_0^1L(s(t),y-x)+H(s(t),d_{s(t)}u)\;dt\leq \calA_{L+h}(s)\,.
\]
Our assumption may not be satisfied.
Moreover,
it could even happen that all the segment
is outside the set $\calD$ in which the derivatives of $u$ exist. 
This happens for instance if $x$ and $y$ are configurations
with collisions and with the same colliding bodies.
However Fubini's Theorem say us that our assumption is verified
for almost every $y\in S_r=\set{y\in E^N\mid \norm{y-x}=r}$.
Then 
\[
u(y)-u(x)\leq \calA_{L+h}(s)\quad\text{ for almost }y\in S_r
\]
Taking into account that both $u$ and $\calA_{L+h}(s)$
are continuous as functions of $y$,
we conclude that the previous inequality holds in fact,
for all $y\in S_r$.

We remark that the same argument applies
to any segment with constant speed,
that is to say, to any curve $s(t)=x+tv$, $t\in [a,b]$.
Concatenating these segments we deduce that the inequality
also holds for any piecewise affine curve
$p\in\calC(x,y)$. The proof is then achieved as follows.

Let $\gamma\in\calC(x,y)$ be a curve such that
$\calA_{L+h}(\gamma)=\phi_h(x,y)$.
The existence of such a curve is guaranteed by Lemma
\ref{lema-JM.geod.complet}.
Since this curve is a minimizer of the Lagrangian action,
Marchal's Theorem assures that,
if $\gamma$ is defined on $[a,b]$,
then $\gamma(t)\in\Omega$ for all $t\in(a,b)$.
In consequence, the restriction of $\gamma$
to $(a,b)$ must be a true motion.

Suppose that there are no collisions at configurations
$x$ and $y$.
Since in this case $\gamma$ is thus $C^1$ on $[a,b]$,
we can approximate it by sequence of piecewise affine curves
$p_n\in\calC(x,y)$,
in such a way that $\dot p_n(t)\to\dot\gamma(t)$
uniformly for $t$ over some full measure subset $D\subset [a,b]$.
In order to be explicit, let us define for each $n>0$
the polygonal $p_n$ with vertices at configurations
$\gamma(a+k(b-a)n^{-1})$ for $k=0,\dots,n$.
Then $D$ can be taken as the complement in $[a,b]$
of the countable set $a+\Q(b-a)$. 
Therefore, we have
$u(y)-u(x)\leq \calA_{L+h}(p_n)$ for all $n\geq 0$,
as well as
\[
\lim_{n\to \infty}\calA_{L+h}(p_n)=
\calA_{L+h}(\gamma)=\phi_h(x,y)\,,
\]
This implies that $u(y)-u(x)\leq \phi_h(x,y)$.
If there are collisions at $x$ or $y$,
then we apply what we have proved
to the configurations without collisions
$x_\epsilon=\gamma(a+\epsilon)$ and
$y_\epsilon=\gamma(b-\epsilon)$,
and we get the same conclusion
taking the limit as $\epsilon\to 0$. 
This proves that $u\prec L+h$.
\end{proof}

\begin{remark}
The use of Marchal's Theorem in the last proof
seems to be required by the argument.
In fact, the argument works well for non singular Hamiltonians
for which it is known \emph{a priori} that the corresponding
minimizers are of class $C^1$. 
\end{remark}

\begin{notation}
We will denote $\calS_h$
the set of viscosity subsolutions of (\ref{HJh}).
\end{notation}

Observe that,
not only we have proved that $\calS_h$ is precisely the set
of dominated functions $u\prec L+h$,
but also that $\calS_h$ agrees with the set of functions
satisfying $H(x,d_xu)\leq h$ almost everywhere in $E^N$. 

\subsubsection{Estimates for the action potentials}

We give now an estimate for $\phi_h$ which implies that
the viscosity subsolutions form an equicontinuous family
of functions.
Therefore, if we normalize subsolutions by imposing $u(0)=0$,
then according to the Ascoli's Theorem
we get to the compactness of the set of normalized subsolutions.

The estimate will be deduced from the basic estimates for
$\phi(x,y,\tau)$ and $\phi(x,y)$
found by the first author for homogeneous potentials
and that we recall now.
They correspond in the reference
to Theorems 1 \& 2 and Proposition 9,
considering that in the original formulation
the value $\kappa=1/2$ is for the Newtonian potential.

We will say that a given configuration $x=(r_1,\dots,r_N)$
is contained in a ball of radius $R>0$ if we have
$\norm{r_i-r_0}_E<R$ for all $1\leq i \leq N$
and for some $r_0\in E$.

\begin{theorem*}[\cite{Mad1}]
There are positive constants $\alpha_0$ and $\beta_0$ such that,
if $x$ and $y$ are two configurations contained
in the same ball of radius $R>0$, then for any $\tau>0$
\[\phi(x,y,\tau)\leq \;
\alpha_0\; \frac{\;R^2}{\tau}\;+\;\beta_0\;\frac{\;\tau}{R}\;.\]
\end{theorem*}

If a configurations $y$ is close enough
to a given configuration $x$,
the minimal radius of a ball containing both configurations
is greater than $\norm{x-y}$. 
However,
this result was successfully combined with an argument
providing suitable cluster partitions,
in order to obtain the following theorem.

\begin{theorem*}[\cite{Mad1}]
There are positive constants $\alpha_1$ and $\beta_1$ such that,
if $x$ and $y$ are any two configurations, and $r>\norm{x-y}$,
then for all $\tau>0$
\begin{equation}
\tag{*}
\label{ineq-basic}\phi(x,y,\tau)\leq \;
\alpha_1\; \frac{\;r^2}{\tau}\;+\;\beta_1\;\frac{\tau}{r}\;.
\end{equation}
\end{theorem*}

Note that the right side of the inequality
is continuous for $\tau,\rho>0$.
Therefore,
we can replace $r$ by $\norm{x-y}$ whenever $x\neq y$.

\begin{remark}
\label{rmk-bound.phixxt}
If $x=y$ then the upper bound (\ref{ineq-basic})
holds for every $r>0$.
Choosing $r=\tau^{2/3}$, we get to the upper bound
$\phi(x,x,\tau)\leq\mu\,\tau^{1/3}$ which holds for any $\tau>0$,
any $x\in E^N$,
and for the positive constant $\mu=\alpha_1+\beta_1$.
\end{remark}
 
Therefore we can now bound the critical potential.
The previous remark leads to $\phi(x,x)=0$ for all $x\in E^N$.
On the other hand,
for the case $x\neq y$ we can bound $\phi(x,y)$
with the bound for $\phi(x,y,\tau)$,
taking $r=\norm{x-y}$ and $\tau=\norm{x-y}^{3/2}$.

\begin{theorem*}
[Hölder estimate for the critical action potential, \cite{Mad1}]
There exist a positive constant $\eta>0$
such that for any pair of configurations
$x,y \in E^N$
\[\phi(x,y)\leq \;\eta\;\norm{x-y}^\frac{1}{2}\;.\]
\end{theorem*}

These estimates for the action potentials have been used
firstly to prove the existence of parabolic motions
\cite{Mad1,MadVen} and were the starting point
for the study of free time minimizers \cite{daLMad,Mad2},
as well as their associated Busemann functions
by Percino and Sánchez \cite{Per, PerSan},
and later by Moeckel, Montgomery and Sánchez \cite{MoMoSa}
in the planar three-body problem.
 
For our current purposes,
we need to generalize the Hölder estimate
of the critical action potential
in order to also include supercritical potentials.
As expected, the upper bound we found
is of the form $\xi(\norm{x-y})$,
where $\xi:[0,+\infty)\to\R^+$
is such that $\xi(r)\approx r^\frac{1}{2}$
for $r\to 0$ and $\xi(r)\approx r$ for $r\to +\infty$.

\begin{theorem}
\label{thm-phih.estim}
There are positive constants $\alpha$ and $\beta$ such that,
if $x$ and $y$ are any two configurations and $h\geq 0$, then
\[
\phi_h(x,y)\leq \;
\left(\alpha\norm{x-y}+h\;\beta\norm{x-y}^2\right)^{1/2}\;.
\]
\end{theorem}

\begin{proof}
We have to bound
$\phi_h(x,y)=
\inf\set{\phi(x,y,t)+h\tau\mid\tau>0}$.
Fix any two configurations $x$ and $y$ and let $r>\norm{x-y}$.
We already know by (\ref{ineq-basic})
that for any $\tau>0$ we have
\begin{equation}
\tag{**}
\label{ineq-basic.h}
\phi(x,y,\tau) + h\tau\;\leq \;
A\; \frac{1}{\tau}\;+B\;\tau
\end{equation}
\[A=\alpha_1\,r^2\quad\text{ and }\quad B=\beta_1\,r^{-1} + h\,,\]
$\alpha_1$ and $\beta_1>0$ being two positive constants.
Since the minimal value of the right side of inequality
(\ref{ineq-basic.h})
as a function of $\tau$ is $2(AB)^{1/2}$ we conclude that
\begin{eqnarray*}
\phi_h(x,y)&=&\inf\set{\phi(x,y,t)+h\tau\mid\tau>0}\\
&\leq&\left(\alpha\,r+h\;\beta\,r^2\right)^{1/2}
\end{eqnarray*}
for
$\alpha=4\,\alpha_1\beta_1$ and $\beta=4\,\alpha_1$.
By continuity,
we have that the last inequality also holds for $r=\norm{x-y}$
as we wanted to prove.
\end{proof}

\begin{corollary} \label{comp-visc-subsol}
\label{coro-visc.ssol.comp}
The set of viscosity subsolutions
$\calS^0_h=\set{u\in\calS_h\mid u(0)=0}$
is compact for the topology of the uniform convergence
on compact sets.
\end{corollary}

\begin{proof}
By Propositions \ref{prop-dom.are.visc.ssol} and
\ref{prop-visc.ssol.are.dom}
we know that $u\in\calS_h$ if and only if $u\prec L+h$.
Thus by Theorem \ref{thm-phih.estim} we have that,
for any $u\in\calS_h$, and for all $x,y\in E^N$, 
\[u(x)-u(y)\leq \phi_h(x,y)\leq \xi(\norm{x-y})\]
where $\xi:[0,+\infty)\to\R^+$ is given by
$\xi(\rho)=\left(\alpha\,\rho+h\,\beta\,\rho^2\right)^{1/2}$.

Since $\xi$ is uniformly continuous,
we conclude that the family of functions $\calS_h$
is indeed equicontinuous.
Therefore,
the compactness of $\calS^0_h$ is actually
a consequence of Ascoli's Theorem.
\end{proof}


\subsection{The Lax-Oleinik semigroup}

We recall that a solution of $H(x,d_xu)=h$
corresponds to a stationary solution
$U(t,x)=u(x)-ht$ of the evolution equation
\[\partial_tU+H(x,\partial_xU)=0\,,\]
for which the Hopf-Lax formula reads
\[U(t,x)=
\inf\set{u_0(y)+\calA_L(\gamma)\mid
y\in E^N,\,\gamma\in\calC(y,x,t)}\,.\]
In a wide range of situations,
this formula provides the \emph{unique}
viscosity solution satisfying the initial condition $U(0,x)=u_0(x)$.
Using the action potential we can also write the formula as
\[U(t,x)=\inf\set{u_0(y)+\phi(y,x,t)\mid y\in E^N}.\]
If the initial data $u_0$ is bounded,
then $U(t,x)$ is clearly well defined and bounded.
In our case, we know that solutions will not be bounded,
thus we need a condition ensuring that the function
$y\mapsto u_0(y)+\phi(y,x,t)$ is bounded by below.
Assuming $u_0\prec L+h$ we have the lower bound
\[u_0(x)-ht \leq u_0(y)+\phi(y,x,t)\]
for all $t>0$  and all $x\in E^N$,
but this is in fact an equivalent formulation
for the domination condition
$u_0\prec L+h$, that is to say $u\in\calS_h$.

\begin{definition}
[Lax-Oleinik semigroup] The backward\footnote{
The \emph{forward} semigroup is defined in a similar way,
see \cite{Eva2}.
This other semigroup gives the opposite solutions
of the reversed Hamiltonian $\tilde{H}(x,p)=H(x,-p)$.
In our case the Hamiltonian is reversible,
meaning that $\tilde{H}=H$.}
Lax-Oleinik semigroup is the map
$T:[0,+\infty)\times \calS_h\to\calS_h$,
given by $T(t,u)=T_tu$, where
\[
T_tu(x)=\inf\set{u(y)+\phi(y,x,t)\mid y\in E^N}
\]
for $t>0$, and $T_0u=u$.
\end{definition}

Observe that $u\prec L+h$ if and only if
$u\leq T_tu+ht$ for all $t>0$.
Also,
we note that $T_tu-u\to 0$ as $t\to 0$, uniformly in $E^N$.
This is clear since for all $x\in E^N$ and $t>0$ we have
$-ht\leq T_tu(x)-u(x)\leq \phi(x,x,t)\leq \mu\,t^{1/3}$,
where the last inequality is justified by Remark
\ref{rmk-bound.phixxt}.

It is not difficult to see that $T$ defines an action on $\calS_h$,
that is to say, that the semigroup property $T_t\circ T_s=T_{t+s}$
is always satisfied.
Thus the continuity at $t=0$ spreads throughout all the domain.

Other important properties of this semigroup are the
\emph{monotonicity},
that is to say, that $u\leq v$ implies $T_tu\leq T_tv$,
and the \emph{commutation with constants},
saying that for every constant $k\in\R$,
we have $T_t(u+k)=T_tu+k$.

Thus,
for $u\in\calS_h$ and $s,t>0$ we can write
$T_su\leq T_s(T_tu+ht)=T_t(T_su)+ht$,
which implies that we have $T_su\in\calS_h$ for all $s>0$.

\begin{definition}
[Lax-Oleinik quotient semigroup]
The semigroup $(T_t)_{t\geq 0}$
defines a semigroup $(\hat{T}_t)_{t\geq 0}$
on the quotient space $\hat{\calS}_h=\calS_h/\R$,
given by $\hat{T}_t[u]=[T_tu]$.
\end{definition}

\begin{proposition}
\label{prop-LO.fixed.points}
Given $h\geq 0$ and $u\in\calS_h$ we have that,
$[u]\in\hat{\calS}_h$ is a fixed point of $(\hat{T}_t)_{t\geq 0}$
if and only if
there is $h'\in [0,h]$ such that $T_tu=u-h't$ for all $t\geq 0$.
\end{proposition}

\begin{proof}
The sufficiency of the condition is trivial.
It is enough then to prove that it is necessary.
That $[u]$ is a fixed point of $\hat{T}$ means that we have
$\hat{T}_t[u]=[u]$ for all $t>0$.
That is to say, there is a function $c:\R^+\to\R^+$ such that
$T_tu=u+c(t)$ for each $t\in\R^+$.
From the semigroup property, we can easily deduce that
the function $c(t)$ must be additive,
meaning that $c(t+s)=c(t)+c(s)$
for all $t,s\geq 0$.
Moreover,
the continuity of the semigroup implies the continuity of $c(t)$.
As it is well known,
a continuous and additive function from $\R^+$ into itself is linear,
therefore we must have $c(t)=c(1)t$.
Now, since $u\leq T_tu+ht$ for all $t\in \R^+$,
we get $0\leq c(1)+h$.
On the other hand, since $u\prec L-c(1)$
and $\calS_h=\emptyset$ for $h<0$, hence $-c(1)\geq 0$.
We conclude that $c(t)=-h't$ for some $h'\in [0,h]$.
\end{proof}

\subsubsection{Calibrating curves and supersolutions}

We finish this section by relating the fixed points
of the quotient Lax-Oleinik semigroup and
the viscosity supersolutions of (\ref{HJh}).
This relationship is closely linked
to the existence of certain minimizers,
which will ultimately allow us to obtain
the hyperbolic motions we seek.

\begin{definition}
[calibrating curves]
Let $u\in\calS_h$ be a given subsolution.
We say that a curve $\gamma:[a,b]\to E^N$
is an \emph{$h$-calibrating}
curve of $u$,
if $u(\gamma(b))-u(\gamma(a))=\calA_{L+h}(\gamma)$.
\end{definition}

\begin{definition}
[h-minimizers]
A curve $\gamma:[a,b]\to E^N$
is said to be an \emph{$h$-minimizer} if
it verifies $A_{L+h}(\gamma)=\phi_h(\gamma(a),\gamma(b))$.
\end{definition}

\begin{remark}
\label{rmk-hcalib.hmin}
As we have see,
the fact that $u\in\calS_h$ is characterized by $u\prec L+h$.
Therefore for all $x,y\in E^N$ we have
\[
u(x)-u(y)\leq \phi_h(x,y)\leq \calA_{L+h}(\gamma)
\]
for any $\gamma\in\calC(x,y)$.
It follows that every $h$-calibrating curve of $u$
is an $h$-minimizer.
\end{remark}

It is easy to prove that restrictions of $h$-calibrating curves
of a given subsolution $u\in\calS_h$
are themselves $h$-calibrating curves of $u$.
This is also true, and even more easy to see, for $h$-minimizers.
But nevertheless,
there is a property valid for the calibrating curves
of a given subsolution but which is not satisfied in general
by the minimizing curves.
The concatenation of two calibrating curves is again calibrating.

\begin{lemma}
\label{lema-concat.calib}
Let $u\in\calS_h$.
If $\gamma_1\in\calC(x,y)$ and $\gamma_2\in\calC(y,z)$
are both $h$-calibrating curves of $u$,
and $\gamma\in\calC(x,z)$ is a concatenation of
$\gamma_1$ and $\gamma_2$,
then $\gamma$ is also an $h$-calibrating curve of $u$.
\end{lemma}

\begin{proof}
We have $u(y)-u(x)=\calA_{L+h}(\gamma_1)$,
and $u(z)-u(y)=\calA_{L+h}(\gamma_2)$.
Adding both equations we get $u(z)-u(x)=\calA_{L+h}(\gamma)$.
\end{proof}

We give now a criterion for a subsolution to be a viscosity solution.
From here on, a curve defined on a noncompact interval will be said
$h$-calibrating if all its restrictions to compact intervals are too.
In the same way we define $h$-minimizers
over noncompact intervals.

We start by proving a lemma on calibrating curves of subsolutions.
\begin{lemma}
\label{lema-no.test.col}
Let $u\in\calS_h$, and let $\gamma:[a,b]\to E^N$ be an
$h$-calibrating curve of $u$.
If $x_0=\gamma(b)$ is a configuration with collisions,
then there is no Lipschitz function
$\psi$ defined on a neighbourhood of $x_0$
such that $\psi\le u$ and $\psi(x_0)=u(x_0)$.
\end{lemma}

\begin{proof}
Since our system is autonomous,
we can assume without loss of generality that $b=0$.
Thus the $h$-calibrating property of $\gamma$ says that
for every $t\in [a,0]$
\[
\int_t^0 \tfrac{1}{2}\norm{\dot\gamma}^2\,dt\,+
\int_t^0 U(\gamma)\,dt\,+ \,h\abs{t}=
\calA_{L+h}(\gamma\mid_{[t,0]})=u(x_0)-u(\gamma(t))\,.
\]
On the other hand, if $\psi\leq u$ is a $k$-Lipschitz function
on a neighbourhood of $x_0$ such that $\psi(x_0)=u(x_0)$
then we also have, for $t$ close enough to $0$,
\[
u(x_0)-u(\gamma(t))\leq
\psi(x_0)-\psi(\gamma(t))\leq
k\norm{\gamma(t)-x_0}.
\]
Therefore we also have
\[
\int_{t}^0\norm{\dot\gamma}^2\,dt\leq\,2k\norm{\gamma(t)-x_0}.
\]
Now, applying Cauchy-Schwarz we can write
\[
\int_{t}^0\norm{\dot\gamma}\,dt\leq
\abs{t}^{1/2}\left(\int_{t}^0\norm{\dot\gamma}^2\,dt\right)^{1/2}\]
and thus we deduce that
\[
\norm{\gamma(t)-x_0}^2\leq
\left(\int_{t}^0\norm{\dot\gamma}\,dt\right)^2\leq 
\,2k\norm{\gamma(t)-x_0}\,\abs{t}
\]
hence that
\[
\norm{\gamma(t)-x_0}\leq
\,2k\abs{t}.
\]
Finally, since
\[
\int_{t}^0U(\gamma)\,dt\leq u(x_0)-u(\gamma(t))\leq
\,k\norm{\gamma(t)-x_0}
\]
we conclude that
\[
\int_{t}^0U(\gamma)\,dt\leq
\,2k^2\abs{t}.
\]
Therefore, dividing by $\abs{t}$ and taking the limit for $t\to 0$
we get $U(x_0)\leq 2k^2$.
This proves that $x_0$ has no collisions.
\end{proof}

\begin{proposition}
\label{prop-criterion.visc.sol}
If $u\in\calS_h$ is a viscosity subsolution of (\ref{HJh}), and
for each $x\in E^N$ there is at least one $h$-calibrating curve
$\gamma:(-\delta,0]\to E^N$  with $\gamma(0)=x$,
then $u$ is in fact a viscosity solution.
\end{proposition}

\begin{proof}
We only have to prove that $u$ is a viscosity supersolution.
Thus we assume that $\psi\in C^1(E^N)$ and $x_0\in E^N$
are such that $u-\psi$ has a local minimum in $x_0$.
We must prove that $H(x_0,d_{x_0}\psi)\geq h$.
First of all, we exclude the possibility that $x_0$
is a configuration with collisions.
To do this, it suffices to apply Lemma \ref{lema-no.test.col},
taking the locally Lipschitz function $\psi+u(x_0)-\psi(x_0)$.

Let now $\gamma:(-\delta,0]\to E^N$ with $\gamma(0)=x_0$
and $h$-calibrating.
Thus for $t\in(-\delta,0]$
\[
\int_t^0L(\gamma,\dot\gamma)\,dt\,-\,ht
=u(x_0)-u(\gamma(t))
\]
and also, given that $x_0$ is a local minimum of $u-\psi$,
for $t$ close enough to $0$
\[
u(x_0)-u(\gamma(t))\leq
\psi(x_0)-\psi(\gamma(t))\,.
\]
Since $x_0\in\Omega$ and $\gamma$ is a minimizer,
we know that $\gamma$ can be extended beyond $t=0$
as solution of Newton's equation.
In particular $v=\dot\gamma(0)$ is well defined,
and moreover, using the previous inequality we find
\[
d_{x_0}\psi(v)=
\;\lim_{t\to 0^-}\,\frac{\psi(x_0)-\psi(\gamma(t))}{-t}\geq
L(x_0,v)+h
\]
which implies, by Fenchel's inequality,
that $H(x_0,d_{x_0}\psi)\geq h$.
\end{proof}

The following proposition complements the previous one.
It states that under a stronger condition,
the viscosity solution is in addition a fixed point
of the quotient Lax-Oleinik semigroup.

\begin{proposition}
\label{prop-fixedLO.if.calib}
Let $u\in\calS_h$ be a viscosity subsolution of (\ref{HJh}).
If for each $x\in E^N$ there is
an $h$-calibrating curve of $u$, say 
$\gamma_x:(-\infty,0]\to E^N$,
such that $\gamma_x(0)=x$,
then $T_tu=u-ht$ for all $t\geq 0$.
\end{proposition}

\begin{proof}
For each $x\in E^N$, for $t\geq 0$ we have
\[
T_tu(x)-u(x)=\inf\set{u(y)-u(x)+\phi(y,x,t)\mid y\in E^N} 
\]
thus it is clear that $T_tu(x)-u(x)\geq -ht$
since we know that $u\prec L+h$.
On the other hand,
given that $\gamma_x$ is  an $h$-calibrating curve of $u$,
\[
u(x)-u(\gamma_x(-t))=\phi(\gamma_x(-t),x,t)+ht.
\]
Writing $y_t=\gamma_x(-t)$ we have that
$u(y_t)-u(x)+\phi(y_t,x,t)=-ht$
and we conclude that $T_tu(x)-u(x)\leq -ht$.
We have proved that $T_tu=u-ht$ for all $t\geq 0$.
\end{proof}

\begin{remark}
\label{rmk-inverse.sens.lamin}
The formulation of the previous condition can confuse a little,
since the calibrating curves are parametrized on negative intervals.
Here the Lagrangian is symmetric,
thus reversing the time of a curve always preserves the action.
More precisely, 
given an absolutely continuous curve $\gamma:[a,b]\to E^N$,
if we define $\tilde{\gamma}$ on $[-b,-a]$
by $\tilde{\gamma}(t)=\gamma(-t)$,
then $\calA_L(\tilde{\gamma})=\calA_L(\gamma)$.

We can reformulate the calibrating condition of the previous
proposition in this equivalent way:
\emph{For each $x\in E^N$,
there is a curve $\gamma_x:[0,+\infty)\to E^N$ such that
$\gamma_x(0)=x$, and such that
$u(x)-u(\gamma_x(t))=\calA_{L+h}(\gamma_x\mid_{[0,t]})$
for all $t>0$}.
\end{remark}
\begin{remark}
The hypothesis of
Proposition \ref{prop-criterion.visc.sol}
implies the hypothesis of
Proposition \ref{prop-fixedLO.if.calib}.
This is exactly what we do in the proof of
Theorem \ref{thm-horofuns.have.lamin} below.
\end{remark}


\section{Ideal boundary of a positive energy level}

This section is devoted to the construction of global viscosity
solutions for the  Hamilton-Jacobi equations (\ref{HJh}).
The method is quite similar to that developed by Gromov
in \cite{Gro} to compactify locally compact metric spaces
(see also \cite{BaGrSch}, chpt. 3).

\subsection{Horofunctions as viscosity solutions}

The underlying idea giving rise to the construction of horofunctions
is that each point in a metric space $(X,d)$ can be identified 
with the distance function to that point.
More precisely,
the map $X\to C(X)$ which associates to each point $x\in X$
the function $d_x(y)=d(y,x)$
is an embedding such that for all $x_0,x_1\in X$
we have $\max\abs{d_{x_0}(y)-d_{x_1}(y)}=d(x_0,x_1)$.

It is clear that
any sequence of functions $d_{x_n}$ diverges if $x_n\to\infty$,
that is to say,
if the sequence $x_n$ escapes from any compact subset of $X$. 
However, for a noncompact space $X$,
the induced embedding of $X$ into the quotient space
$C(X)/\R$ has in general an image with a non trivial boundary.
This boundary can thus be considered
as an ideal boundary of $X$.

Here the metric space will be $(E^N,\phi_h)$ with $h>0$,
and the set of continuous functions $C^0(E^N)$
will be endowed with
the topology of the uniform convergence on compact sets.
Instead of looking at equivalence classes of functions,
we will take as the representative of each class
the only one vanishing at $0\in E^N$. 

\begin{definition}
[Ideal boundary]
We say that a function $u\in C^0(E^N)$ is in the ideal boundary
of level $h$ if there is a sequence of configurations $p_n$,
with $\norm{p_n}\to +\infty$ and such that for all $x\in E^N$
\[
u(x)=\lim_{n\to\infty}\phi_h(x,p_n)-\phi_h(0,p_n).
\]
We will denote $\calB_h$ the set of all these functions,
that we will also call horofunctions.
\end{definition}

The first observation is that $\calB_h\neq\emptyset$
for any value of $h\geq 0$.
This can be seen as a consequence of the estimate for the
potential $\phi_h$ we proved,
see Theorem \ref{thm-phih.estim}.

Actually for any $p\in E^N$,
the function $x\to \phi_h(x,p)-\phi_h(0,p)$ is in $\calS^0_h$,
the set of viscosity subsolutions vanishing at $x=0$.
Since by Corollary \ref{coro-visc.ssol.comp}
we know that $\calS^0_h$ is compact,
for any sequence of configurations $p_n$
such that $\norm{p_n}\to +\infty$
there is a subsequence
which defines a function in $\calB_h$ as above.

It is also clear that $\calB_h\subset\calS_h$.
Functions in $\calB_h$ are limits of functions in $\calS_h$,
and this set is closed in $E^N$
even for the topology of pointwise convergence.
But, since we already know that
the family $\calS_h$ is equicontinuous,
the convergence is indeed uniform on compact sets.

\begin{notation}
When the value of $h$ is understood,
we will denote $u_p$ the function
defined by $u_p(x)=\phi_h(x,p)$
where $p$ is a given configuration.
\end{notation}

One fact that should be clarifying is that for any $p\in E^N$,
the subsolution given by $u_p$ fails to be a viscosity solution
precisely at $x=p$.
If $x\neq p$, then there is a minimizing curve of $\calA_{L+h}$
in $\calC(p,x)$
(see Lemma \ref{lema-JM.geod.complet} below),
and clearly this curve is $h$-calibrating of $u_p$.
On the other hand, there are no $h$-calibrating curves of $u_p$
defined over an interval $(-\delta,0]$
and ending at $x=p$.
This is because $u_p\geq 0$, $u_p(p)=0$,
and $h$-calibrating curves, as $h$-minimizers,
have strictly increasing action.
Actually, this property of the $u_p$ functions occurs
for all energy levels greater than or equal to the critical one,
in a wide class of Lagrangian systems.
The simplest case to visualize is surely
the case of absence of potential energy in an Euclidean space,
in which we have $u_p(x)=h\,\norm{x-p}$ and
his $h$-calibrating curves are segments of the half-lines
emanating from $p$ with a constant speed (gradient curves).

This suggest that the horofunctions must be viscosity solutions,
which is what we will prove now.

\begin{theorem}
\label{thm-horofuns.are.viscsol}
Given $u\in\calB_h$ and $r>0$ there is,
for each $x\in E^N$, some $y\in E^N$ with $\norm{y-x}=r$,
and a curve $\gamma_x\in\calC(y,x)$ such that
$u(x)-u(y)=\calA_{L+h}(\gamma_x)$.
In particular, every function $u\in\calB_h$
is a global viscosity solution of (\ref{HJh}).
\end{theorem}

\begin{proof}
Let $u\in\calB_h$, that is to say
$u=\lim_n(u_{p_n}-u_{p_n}(0))$
for some sequence of configurations $p_n$ such that
$\norm{p_n}\to +\infty$, and $u_{p_n}(x)=\phi_h(x,p_n)$.

Let $x\in E^N$ be any configuration, and fix $r>0$.
Using Lemma \ref{lema-JM.geod.complet} we get,
for each $n>0$, a curve $\gamma_n\in\calC(p_n,x)$
such that $\calA_{L+h}(\gamma_n)=\phi_h(p_n,x)$.
Each curve $\gamma_n$ is thus
an $h$-calibrating curve of $u_{p_n}$.

If $\norm{p_n-x}>r$,
then the curve $\gamma_n$ must pass through a configuration
$y_n$ with $\norm{y_n-x}=r$.
Extracting a subsequence if necessary,
we may assume that this is the case for all $n>0$,
and that $y_n\to y$, with $\norm{y-x}=r$.
Since the arc of $\gamma_n$ joining
$y_n$ to $x$ also $h$-calibrates $u_{p_n}$ we can write
\[
u_{p_n}(x)-u_{p_n}(y_n)=\phi_h(y_n,x)
\]
for all $n$ big enough. We conclude that
\[
u(x)-u(y)=\lim_{n\to\infty}u_{p_n}(x)-u_{p_n}(y)=\phi_h(y,x)
\]
which proves the first statement.
The second one follows now from the criterion for
viscosity solutions given in Proposition
\ref{prop-criterion.visc.sol}.	
\end{proof}
Our next goal is to prove that horofunctions are actually
fixed points of the quotient Lax-Oleinik semigroup.
We will achieve this goal by showing the existence
of calibrating curves allowing the use of
Proposition \ref{prop-fixedLO.if.calib}.
These calibrating curves will be the key to the proof of
the existence of hyperbolic motions.

Thanks to the previous theorem
we can build maximal calibrating curves.
Then, Marchal's Theorem will allow us to assert that
these curves are in fact true motions of the $N$-body problem.
Next we have to prove
that these motions
are defined over unbounded above time intervals,
that is to say,
we must exclude the possibility of
collisions or pseudocollisions.
It is for this reason that we will also invoke
the famous von Zeipel's theorem\footnote{
This theorem had no major impact on the theory
until it was rediscovered after at least half a century later,
and proved to be essential for the understanding
of pseudocollision singularities, see for instance Chenciner's
Bourbaki seminar \cite{Che2}.
Among other proofs, there is a modern version due to McGehee 
 \cite{McG} of the proof originally outlined by von Zeipel.}
 that we recall now.

\begin{theorem*}
[1908, von Zeipel \cite{Zei}]
Let $x:(a,t^*)\to E^N$ be a maximal solution of the
Newton's equations of the $N$-body problem with $t^*<+\infty$.
If $\norm{x(t)}$ is bounded in some neighbourhood of $t^*$, then
the limit $\,x_c=\lim_{t\to t^*}x(t)$
exists and the singularity is therefore due to collisions. 
\end{theorem*}

\begin{theorem}
\label{thm-horofuns.have.lamin}
If $u\in\calB_h$ then for each $x\in E^N$ there is a curve
$\gamma_x:[0,+\infty)\to E^N$ with $\gamma_x(0)=x$,
and such that for all $t>0$
\[
u(x)-u(\gamma_x(t))=\calA_{L+h}(\gamma_x\mid_{[0,t]}).
\]
In particular,
every function $u\in\calB_h$ satisfies $T_tu=u-ht$ for all $t>0$.
\end{theorem}

\begin{proof}
Let us fix a configuration $x\in E^N$.
By Theorem \ref{thm-horofuns.are.viscsol} we know that
$u$ has at least one $h$-calibrating curve
$\gamma:(-\delta,0]\to E^N$ such that $\gamma(0)=x$.
By application of Zorn's Lemma
we get a maximal $h$-calibrating curve
of the form $\gamma:(t^*,0]\to E^N$ with $\gamma(0)=x$.
We will prove that $t^*=-\infty$,
and thus the required curve can be defined on $[0,+\infty)$ by
$\gamma_x(t)=\gamma(-t)$.

Suppose by contradiction that $t^*>-\infty$.
Since $\gamma$ is an $h$-minimizing curve, we know that its
restriction to $(t^*,0)$ is a true motion with energy constant $h$.
Either the curve can be extended as a motion for values
less than $t^*$, or it presents a singularity at $t=t^*$.
In the case of singularity, we have at $t=t^*$ either a collision,
or a pseudocollision.
According to von Zeipel's Theorem,
in the pseudocollision case we must have
$\sup\set{\norm{\gamma(t)}\mid t\in (t^*,0]}=+\infty$.

Suppose that the limit $y=\lim_{t\to t^*}\gamma(t)$ exists.
Then by Theorem \ref{thm-horofuns.are.viscsol} we can choose
a calibrating curve $\tilde{\gamma}$ defined on $(-\delta,0]$
and such that $\tilde{\gamma}(0)=y$.
Thus the concatenation of $\tilde{\gamma}$ with $\gamma$
defines a calibrating curve $\gamma^+$ defined on
$(t^*-\delta,0]$ and such that $\gamma^+(0)=x$.
But this contradicts the maximality of $\gamma$.

On the other hand, if we suppose that $\norm{\gamma(t)}$
is unbounded, we can choose a sequence $y_n=\gamma(t_n)$
such that $\norm{y_n-x}\to +\infty$.
Let us define $A_n=\calA_L(\gamma\mid_{[t_n,0]})$.

A standard way to obtain a lower bound for $A_n$
is by neglecting the potential term which is positive.
Then by using  the Cauchy-Schwarz inequality we obtain that
for all $n>0$ we have $2\abs{t_n}A_n\geq \norm{y_n-n}^2$.
Since $\gamma$ is $h$-minimizing we deduce that
\[
\phi_h(y_n,x)\geq
\frac{\norm{y_n-x}^2}{2\abs{t_n}}+h\abs{t_n}
\]
for all $n>0$.
Since $\norm{y_n-x}\to +\infty$ and $t_n\to t^*>-\infty$
we get a contradiction with the upper estimate given by Theorem
\ref{thm-phih.estim}.
Indeed that theorem implies that
$\phi_h (y_n,x)$ is bounded above by a function which is
of order $O(\norm{y_n - x})$ as $\norm{y_n-x} \to +\infty$,  
which contradicts the displayed inequality.

The last assertion is a consequence of
Proposition \ref{prop-fixedLO.if.calib} and
Remark \ref{rmk-inverse.sens.lamin}.
\end{proof}

\subsection{Busemann functions}
\label{s-busemann}

We recall that a length space $(X,d)$ is say to be a
\emph{geodesic space}
if the distance between any two points is realized as the length
of a curve joining them.
A \emph{ray} in $X$ is an isometric embedding
$\gamma:[0,+\infty)\to X$.
As we already say in Sect. \ref{s-geom.view},
the Gromov boundary of a geodesic space is defined as
the quotient space of the set of rays of $X$ under the
equivalence relation: $\gamma\sim\gamma'$ if and only if the
function given by $d(\gamma(t),\gamma'(t))$ on $[0,+\infty)$
is bounded.

There is a natural way to associate a horofunction to each ray.
Let us write $d_p$ for the function measuring the distance to
the point $p\in X$ that is, $d_p(x)=d(x,p)$.
Once $\gamma$ is fixed, we define
\[
u_t(x)=d_{\gamma(t)}(x)-d_{\gamma(t)}(\gamma(0))\,,
\quad\text{ and }\quad
u_\gamma=\lim_{t\to\infty}u_t\,.
\]

These horofunctions $u_\gamma$ are called
\emph{Busemann functions}
and are well defined because of
the geodesic characteristic property of rays.
More precisely, for any ray $\gamma$ and for all $0\leq s\leq t$,
we have $d(\gamma(t),\gamma(s))=t-s$,
hence that $u_t\leq u_s$.
Moreover, 
it is also clear that $u_t\geq -d_{\gamma(0)}$,
which implies that $u_\gamma\ge -d_{\gamma(0)}$.
We also note that $u_\gamma=\lim_n u_{t_n}$ whenever
$(t_n)_{n>0}$ is a sequence such that $t_n\to\infty$.

It is well known that under some hypothesis on $X$ we have that,
for any two equivalent rays $\gamma\sim\gamma'$,
the corresponding Busemann functions are the same
up to a constant, that is $[u_\gamma]=[u_{\gamma'}]$.
Therefore in these cases a map is defined
from the Gromov boundary into the ideal boundary,
and it is thus natural to ask about the injectivity and the surjectivity
of this map.
However, the following simple and enlightening example
shows a geodesic space in which there are equivalent rays
$\gamma\sim\gamma'$ for which
$[u_\gamma]\neq [u_{\gamma'}]$.

\begin{example}
[The infinite ladder]
We define $X\subset\R^2$ as the union of the two straight lines
$\R\times\set{-1,1}$ with the segments $\Z\times [-1,1]$,
see figure \ref{ladder}. 

\begin{figure}[h]
\centering
\includegraphics[scale=1]{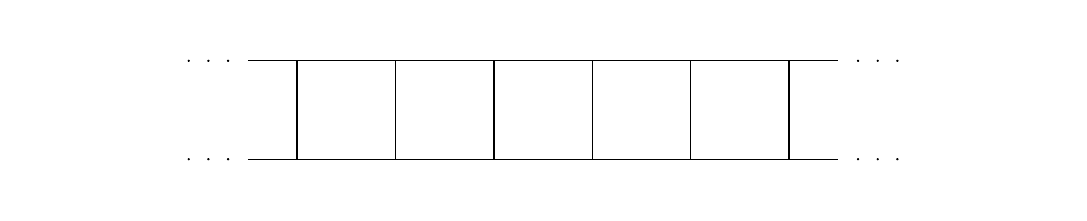}
\caption{The infinite ladder.} 
\label{ladder}
\end{figure}

We endow $X$ with the length distance induced by the standard
metric in $\R^2$.
It is not difficult to see that every ray in $X$
is eventually of the form
$x(t)=(\pm t+c, \pm 1)$.
Each possibility for the two signs determines one of the four
different Busemann functions which indeed compose the
ideal boundary.
Therefore,
 there are four points in the ideal boundary of $X$,
while there is only two classes of rays
composing the Gromov boundary of $X$.
\end{example}

Let us return to the context of the $N$-body problem,
that is to say,
let us take as metric space the set of configurations $E^N$,
with the action potential $\phi_h$ as the distance function.
Actually $(E^N,\phi_h)$ becomes a length space, and $\phi_h$
coincides with the length distance of the Jacobi-Maupertuis
metric when restricted to $\Omega$.
Proofs of all these facts are given in Sect. \ref{s-jm.dist}.
We are interested in the study of the ideal and Gromov
boundaries of this space,
in particular we need to understand the rays in
this space having prescribed asymptotic direction.
As we will see, they will be found as calibrating curves
of horofunctions in a special class.

\begin{definition}
[Directed horofunctions]
Given a configuration $a\neq 0$ we define the set
of horofunctions directed by $a$ as the set
\[
\calB_h(a)=
\set{u\in\calB_h\mid
u=\lim_n(u_{p_n}-u_{p_n}(0))\,,\;
p_n=\lambda_na+o(\lambda_n),\;
\lambda_n\to +\infty}.
\]
\end{definition}
\begin{remark}
Theorem \ref{thm-phih.estim} implies,
in a manner identical to the proof of
Corollary \ref{comp-visc-subsol},
that $\calB_h(a)\neq\emptyset$.
\end{remark}
  
\begin{figure}[h]
\centering
\includegraphics[scale=1.1]{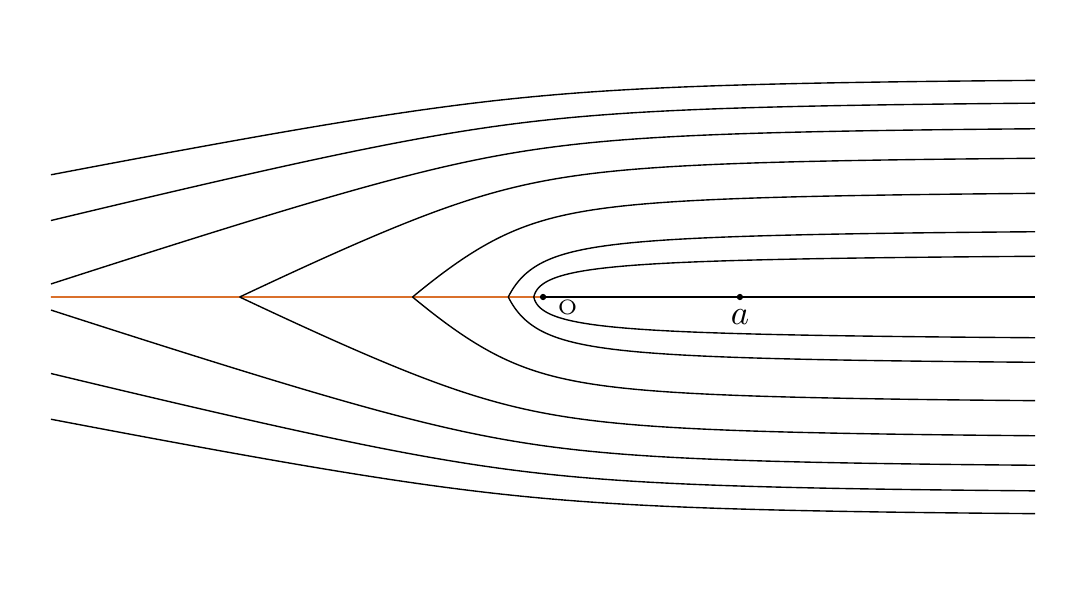}
\caption{Calibrating curves of a hyperbolic Busemann function
$u(x)=\lim_n (\phi_h(x,na)-\phi_h(0,na))$ in the Kepler problem.} 
\label{HyperKepler2}
\end{figure}

The following theorem is the key
for the proof of Theorem \ref{thm-princ}
and its proof is given in Sect. \ref{s-proofs.main}.

\begin{theorem}
\label{thm-calib.direct.horo}
Let $a\in\Omega$ and $u\in\calB_h(a)$.
If $\gamma:[0,+\infty)\to E^N$ satisfies
\[
u(\gamma(0))-u(\gamma(t))=\calA_{L+h}(\gamma\mid_{[0,t]})
\]
for all $t>0$,
then $\gamma$ is a hyperbolic motion of energy $h$
with asymptotic direction $a$.
\end{theorem}

We can thus deduce the following corollary,
whose proof is a very easy application of the Chazy's Theorem on
hyperbolic motions, see Remark \ref{rmk-Chazy.implic}.
\begin{corollary}
\label{coro-calib.bus.equiv}
If $a\in\Omega$ and $u\in \calB_h(a)$ then
the distance between any two $h$-calibrating curves for $u$
is bounded on their common domain.
\end{corollary}

We can also apply Theorem \ref{thm-calib.direct.horo} to deduce
that calibrating curves of a hyperbolic Busemann function are
mutually asymptotic hyperbolic motions.
\begin{corollary}
\label{coro-calib.of.busemann}
If $\gamma$ is an hyperbolic $h$-minimizer,
and $u_\gamma$ its associated Busemann function,
then all the calibrating curves of $u_\gamma$ are hyperbolic
motions with the same limit shape and direction as $\gamma$.
\end{corollary}

\begin{proof}
Since $\gamma$ is hyperbolic, we know that
there is a configuration without collisions $a\in\Omega$
such that $\gamma(t)=ta+o(t)$ as $t\to +\infty$.
Taking the sequence 
$p_n=\gamma(n)$ we have that $p_n=na+o(n)$,
and also that
\[
u_\gamma-u_\gamma(0)=
\lim_{n\to +\infty}[u_{p_n}-u_{p_n}(0)].
\]
This implies that $u_\gamma-u_\gamma(0)\in\calB_h(a)$,
hence that $u_\gamma$ is a viscosity solution and
moreover, Theorem \ref{thm-calib.direct.horo} says that
the calibrating curves of $u_\gamma$ all of the form $ta+o(t)$.
On the other hand, clearly $\gamma$ calibrates $u_\gamma$
since for any $0\leq s \leq t$ we have that
\[
u_{\gamma(t)}(\gamma(s))-u_{\gamma(t)}(\gamma(0))=
-\phi_h(\gamma(0),\gamma(s)),
\]
which in turn implies, taking the limit for $t\to +\infty$, that
\[
u_\gamma(\gamma(0))-u_\gamma(\gamma(s))=
-u_\gamma(\gamma(s))=
\phi_h(\gamma(0),\gamma(s)).\]
\end{proof}

\section{Proof of the main results on hyperbolic motions}

This part of the paper contains the proofs that so far it has been
postponed for different reasons.
In the first part we deal with several lemmas and technical results,
after which we complete the proof of the main results in Sect.
\ref{s-proofs.main}.

\subsection{Chazy's Lemma}
\label{s-Chazy.lema}

The first lemma that we will prove states that the set
$\calH^+\subset T\Omega$ of initial conditions in the
phase space given rise to hyperbolic motions is an open set.
Moreover,
it also says that the map defined in this set
which gives the asymptotic velocity in the future
is continuous. 
This is precisely what in Chazy's work appears as
\emph{continuité de l'instabilité}.
We give a slightly more general version for homogeneous
 potentials of degree $-1$,
but the proof works the same for potentials
of negative degree in any Banach space.

Intuitively what happens is that,
if an orbit is sufficiently close to some given hyperbolic motion,
then after some time the bodies will be so far away each other,
that the action of the gravitational forces
will not be able to perturb their velocities too much.

\begin{lemma}
\label{lema-cont.limitshape}
Let $U:E^N\to\R\cup\set{+\infty}$ be a homogeneous potential
of degree $-1$ of class $C^2$ on the open set
$\Omega=\set{x\in E^N\mid U(x)<+\infty}$.
Let $x:[0,+\infty)\to\Omega$ be a given solution of
$\ddot x=\nabla U(x)$ satisfying $x(t)=ta+o(t)$
as $t\to +\infty$ with $a\in\Omega$.

Then we have:
\begin{enumerate}
\item[(1)] The solution $x$ has asymptotic velocity $a$,
meaning that
\[
\lim_{t\to+\infty}\dot x(t)=a\,.
\]
\item[(2)] (Chazy's continuity of the limit shape)
Given $\epsilon>0$,
there are constants $t_1>0$ and $\delta>0$ such that,
for any maximal solution $y:[0,T)\to\Omega$ satisfying
$\norm{y(0)-x(0)}<\delta$ and
$\norm{\dot y(0)-\dot x(0)}<\delta$,
we have:
\begin{enumerate}
\item[(i)] $T=+\infty$, $\norm{y(t)-ta}<t\epsilon$ for all $t>t_1$,
and moreover
\item[(ii)] there is $b\in\Omega$ with $\norm{b-a}<\epsilon$ 
for which $y(t)=tb+o(t)$.
\end{enumerate}
\end{enumerate}
\end{lemma}

\begin{proof}
Let $0<\rho<\epsilon$ such that the closed ball
$B=\overline B(a,\rho)$ is contained in $\Omega$.
Let $k=\max \set{\norm{\nabla U(z)}\mid z\in B}$, and choose
$t_0>0$ in such a way that for any $t\geq t_0$ we have
$\norm{x(t)-ta}<t\rho$. Therefore, since $\nabla U$ is
homogeneous of degree $-2$, for each $t\geq t_0$ we have
$t^{-1}x(t)\in B$ and
\[
\norm{\nabla U(x(t))} \leq
t^{-2}\norm{\nabla U(t^{-1}x(t))} \leq
kt^{-2}\,.
\]
Thus, for $t_0<t_1<t_2$ we can write
\[
\norm{\dot x(t_2)-\dot x(t_1)} \leq
\int_{t_1}^{t_2}\norm{\nabla U(x(s))}ds \leq
\int_{t_1}^{+\infty}ks^{-2}\,ds =
\frac{k}{t_1}
\]
from which we deduce that $\dot x(t)$ has a limit for $t\to +\infty$.
This limit can not be other than $\lim \dot x= a$,
since otherwise we would have that the
derivative of $x(t)-ta$ has a non null limit contradicting
the hypothesis $x(t)-ta=o(t)$.

Writing $x_1=x(t_1)$ and $\dot x_1=\dot x(t_1)$, we see that
we can fix $t_1>t_0$ large enough such that
\[
\norm{x_1-t_1a}<t_1\,\frac{\rho}{3}\,.
\]
If in addition $t_1>3k/\rho$ we also have
\[
\norm{\dot x_1-a}\leq \frac{k}{t_1}<\frac{\rho}{3}\,.
\]

On the other hand, since the vector field $X(x,v)=(v,\nabla U(x))$
is of class $C^1$, it defines a local flow on $T\Omega$.
Let us denote by $(x_0,\dot x_0)$ the initial condition
$(x(0),\dot x(0))$ of $x(t)$.
We can choose $\delta>0$ such that, for any choice of
$(y_0,\dot y_0)\in T\Omega$ verifying
$\norm{y_0-x_0}<\delta$ and $\norm{\dot y_0-\dot x_0}<\delta$,
the maximal solution $y:[0,T)\to\Omega$
with $y(0)=y_0$  and $\dot y(0)=\dot y_0$
satisfies the following two conditions: $T>t_1$, and
\[
\norm{y_1-t_1a}<t_1\,\frac{\rho}{3},
\;\textrm{ and }\;
\norm{\dot y_1-a}<\frac{\rho}{3}
\]
where $y_1=y(t_1)$ and $\dot y_1=\dot y(t_1)$.

Now, assume that $t\in[t_1,T)$ is such that $y(s)\in sB$
for all $s\in[t_1,t]$.
As before we have $\norm{\dot y(s)-\dot y_1}\leq k/t_1<\rho/3$,
and thus $\norm{\dot y(s)-a}<2\rho/3$.
Therefore we can deduce that
\begin{eqnarray*}
\norm{y(t)-ta}&\leq& \norm{y_1-t_1a}
                              +\int_{t_1}^t\norm{\dot y(s)-a}\,ds\\
 &<& t_1\,\frac{\rho}{3} +(t-t_1)\frac{2\rho}{3}<t\rho
\end{eqnarray*}
Since the last inequality is strict, in fact we have proved that
$y(s)\in s\interior{B}$ for all $s\in [t_1,t]$,
where $\interior{B}$ denotes the open ball $B(a,\rho)$.
Thus, the set of $t\in[t_1,T)$ such that $y(s)\in sB$
for all $s\in[t_1,t]$ is an open subset, and easily we conclude
that we must have $y(t)\in t\interior{B}$ for all $t\in[t_1,T)$.

$T=+\infty$.
Otherwise $K=\cup_{t\in[0,T]}\,tB$ would be compact and
$(y(t),\dot y(t))\in K\times B$ for all $t\in[t_1,T)$,
which is impossible for a maximal solution.
By the same argument used for the motion $x$,
we have that $\dot y(t)$ has a limit $b\in B$.
In particular $\norm{b-a}<\epsilon$ and $y(t)=tb+o(t)$.
\end{proof}

\subsection{Existence and properties of h-minimizers}

The following lemma ensures that for $h>0$,
the length space $(E^N,\phi_h)$ is indeed geodesically convex.

Actually the lemma give us minimizing curves
for any pair of configurations, even with collisions,
and it follows from Marchal's Theorem that
such curves avoid collisions at intermediary times.
The proof is a well-known argument based on the
Tonelli's Theorem for convex Lagrangians,
combined with Fatou's Lemma
for dealing with the singularities of the potential.

\begin{lemma}[Existence of minimizers for $\phi_h$]
\label{lema-JM.geod.complet}
Given $h>0$ and $x\neq y\in E^N$ there is a curve
$\gamma\in\calC(x,y)$  such that
$\calA_{L+h}(\gamma)=\phi_h(x,y)$.
\end{lemma}

We need to introduce before some notation
and make a simple remark that we will use several times.
It is worth noting that the remark applies whenever we consider
a system defined by a potential $U>0$.

\begin{notation}
Given $h\geq 0$, for $x,y\in E^N$ and $\tau>0$ we will write
\[
\Phi_{x,y}(\tau)=\tfrac{1}{2}\norm{x-y}^2\tau^{-1}+h\,\tau\,.
\]
\end{notation}

\begin{remark}
\label{rmk-usual.bound}
Given $h\geq 0$ we have, for any pair of configurations
$x,y\in E^N$ and any $\tau>0$
\[
\phi(x,y,\tau)+ h \tau \;\geq\; \Phi_{x,y}(\tau).
\]
Indeed, given any pair of configurations
$x,y\in E^N$ and for any $\sigma\in\calC(x,y,\tau)$,
the Cauchy-Schwarz inequality implies
\[
\norm{x-y}^2\leq
(\;\int_a^b\norm{\dot \sigma}\,dt\;)^2\leq
\;\tau\,\int_a^b\norm{\dot \sigma}^2\,dt\,,
\]
thus, since $U>0$,
\[
\calA_L(\sigma)>
\tfrac{1}{2}\int_a^b\norm{\dot \sigma}^2\,dt \geq
\tfrac{1}{2}\,\norm{x-y}^2\tau^{-1}.
\]
This justifies the assertion,
since this lower bound does not depend on the curve $\sigma$.
\end{remark}

\begin{proof}
[Proof of Lemma \ref{lema-JM.geod.complet}]
Let $x,y\in E^N$ be two given configurations, with $x\neq y$.
We start by taking a minimizing sequence of
$\calA_{L+h}$ in $\calC(x,y)$, that is to say,
a sequence of curves $(\sigma_n)_{n> 0}$ such that
\[
\lim_{n\to\infty}\calA_{L+h}(\sigma_n)=\phi_h(x,y)\,.
\]
Then from this minimizing sequence we build a new one,
but this time composed by curves with the same domain.
To do this, we first observe that,
if each $\sigma_n$ is defined on an interval $[0,\tau_n]$,
then by the previous remark we know that
\[
\calA_{L+h}(\sigma_n)\geq
\phi(x,y,\tau_n)+h\tau_n \geq
\Phi_{x,y}(\tau_n)
\]
where $\Phi_{x,y}$ is the above defined function.
Since clearly $\Phi_{x,y}$ is a proper function on $\R^+$,
we deduce that
$0<\liminf \tau_n\leq \limsup \tau_n<+\infty$,
and therefore we can suppose without loss of generality
that $\tau_n\to\tau_0$ as $n\to\infty$.
It is not difficult to see that reparametrizing linearly
each curve $\sigma_n$ over the
interval $[0,\tau_0]$ we get a new minimizing sequence.
More precisely, for each $n>0$ the reparametrization is
the curve
$\gamma_n:[0,\tau_0]\to E^N$ defined by
$\gamma_n(t)=\sigma_n(\tau_n\tau_0^{-1}\, t)$.
Computing the action of the curves $\gamma_n$ we get
\[
\int_0^{\tau_0}\tfrac{1}{2}\norm{\dot\gamma_n}^2\,dt=
\tau_n\tau_0^{-1}
\int_0^{\tau_n}\tfrac{1}{2}\norm{\dot\sigma_n}^2\,dt
\]
and
\[
\int_0^{\tau_0} U(\gamma)\,dt=
\tau_0\tau_n^{-1}
\int_0^{\tau_n} U(\sigma)\,dt
\]
thus we have that
\[
\lim_{n\to\infty}\calA_{L+h}(\gamma)=
\lim_{n\to\infty}\calA_{L+h}(\sigma)=
\phi_h(x,y).
\]

On the other hand, It is easy to see that a uniform bound
for the action of the family of curves $\gamma_n$
implies the equicontinuity of the family.
More precisely,
if the bound $\calA_L(\gamma_n)\leq \tfrac{1}{2}\,M^2$
holds for all $n>0$,
then using Cauchy-Schwarz inequality as in Remark
\ref{rmk-usual.bound} we have
\[
\norm{\gamma_n(t)-\gamma_n(s)}\leq
M\abs{t-s}^\frac{1}{2}
\]
for all $t,s\in [0,t_0]$ and for all $n>0$.
Thus by Ascoli's Theorem we can assume that
the sequence $(\gamma_n)$ converges uniformly to
a curve $\gamma\in\calC(x,y,\tau_0)$.
Finally, we apply Tonelli's Theorem for convex Lagrangians to get
\[
\int_0^{\tau_0}\tfrac{1}{2}\norm{\dot\gamma}^2\,dt\;\leq\;
\liminf_{n\to\infty}
\int_0^{\tau_0}\tfrac{1}{2}\norm{\dot\gamma_n}^2\,dt
\]
and Fatou's Lemma to obtain that
\[
\int_0^{\tau_0} U(\gamma)\,dt\;\leq\;
\liminf_{n\to\infty}
\int_0^{\tau_0} U(\gamma_n)\,dt.
\]
Therefore $\calA_L(\gamma)\leq \phi(x,y,\tau_0)$,
which is only possible if the equality holds,
and thus we deduce that $\calA_{L+h}(\gamma)=\phi_h(x,y)$.
\end{proof}

The next lemma is quite elementary and provides a
rough lower bound for $\phi_h$.
However it has an interesting consequence, namely that
reparametrizations of the $h$-minimizers by arc length
of the metric $\phi_h$ are Lipschitz
with the same Lipschitz constant.
We point out that this lower bound only depends
on the positivity of the Newtonian potential.

\begin{lemma}\label{lema-geod.are.lipschitz}
Let $h>0$.
For any pair of configurations $x,y\in E^N$ we have
\[
\phi_h(x,y)\geq \sqrt{2h}\norm{x-y}.
\]
\end{lemma}

\begin{proof}
We note that
\[
\phi_h(x,y)=
\min\set{\phi(x,y,\tau)+\tau h\mid \tau>0}\geq
\min\set{\Phi_{x,y}(\tau)\mid\tau>0},
\]
and that
\[
\min\set{\Phi_{x,y}(\tau)\mid \tau>0}=
\sqrt{2h}\norm{x-y}.
\]
\end{proof}

\begin{remark}
\label{rmk-reparam.are.Lip}
If $\gamma(s)$ is a reparametrization of an $h$-minimizer
and the parameter is the arc length for the metric $\phi_h$, 
then we have
\[
\sqrt{2h}\,\norm{\gamma(s_2)-\gamma(s_1)}\leq
\phi_h(\gamma(s_1),\gamma(s_2))=\abs{s_2-s_1}.
\]
Therefore all these reparametrizations are Lipschitz
with Lipschitz constant $1/\sqrt{2h}$. 
\end{remark}

Finally, the following and last lemma will be used to estimate
the time needed by an $h$-minimizer to join two given
configurations.

\begin{lemma}
\label{lema-time.estim}
Let $h>0$, $x,y\in E^N$ two given configurations,
and let $\sigma\in\calC(x,y,\tau)$ be an $h$-minimizer.
Then we have
\[
\tau_-(x,y)\leq \tau \leq \tau_+(x,y)
\]
where $\tau_-(x,y)$ and $\tau_+(x,y)$
are the roots of the polynomial
\[
P(\tau)=2h\,\tau^2-2\phi_h(x,y)\,\tau+\norm{x-y}^2.
\] 
\end{lemma}

\begin{proof}
Since $\sigma$ minimizes $A_{L+h}$,
in view of Remark \ref{rmk-usual.bound} we have
\[
\phi_h(x,y)=
\phi(x,y,\tau)+\tau h\geq
\Phi_{x,y}(\tau)
\]
that is,
\[\phi_h(x,y)\geq \frac{\norm{x-y}^2}{2\tau}+\tau h\,,\]
which is equivalent to say that $P(\tau)<0$.
\end{proof}

\subsection{Proof of Theorems \ref{thm-princ} and \ref{thm-calib.direct.horo} }
\label{s-proofs.main}

\begin{proof}[Proof of Theorem \ref{thm-princ}]
Given $h>0$, $a\in\Omega$ and $x_0\in E^N$ we proceed as
follows.
First, we define the sequence of functions
\[
u_n(x)=
\phi_h(x,na)-\phi_h(0,na)\,,\qquad x\in E^N.
\]
Each one of this functions is a viscosity subsolution
of the Hamilton-Jacobi equation $H(x,d_xu)=h$,
that is to say, we have $u_n\prec L+h$ for all $n>0$.
Since the estimate for the action potential $\phi_h$
given by Theorem \ref{thm-phih.estim} implies that
the set of such subsolutions is an equicontinuous family,
and since we have $u_n(0)=0$ for all $n>0$,
we can extract a subsequence converging to a function
\[
\mfu (x)=\lim_{k\to +\infty}u_{n_k}(x),
\]
and the convergence is uniform on compact subsets of $E^N$.
Actually the limit is a directed horofunction
$\mfu\in\calB_h(a)$.

By Theorem \ref{thm-horofuns.have.lamin}
we know that there is at least one curve $x:[0,+\infty)\to E^N$,
such that
\[
\phi_h(x_0,x(t))=
\calA_L(x\mid_{[0,t]})+ht=
\mfu (x_0)-\mfu (x(t)).
\]
for any $t>0$, and such that $x(0)=x_0$.
Proposition \ref{prop-criterion.visc.sol} now implies
that $\mfu$ is a viscosity solution of
the Hamilton-Jacobi equation $H(x,d_xu)=h$, and moreover,
that $\mfu$ is a fixed point of the quotient Lax-Oleinik semigroup.

Finally, by Theorem \ref{thm-calib.direct.horo} we have that
the curve $x(t)$ is a hyperbolic motion, with energy constant $h$,
and whose asymptotic direction is given by the configuration $a$.
More precisely, we have that
\[
x(t)=t\;\frac{\sqrt{2h}}{\norm{a}}\;a \,+\,o(t)
\]
as $t\to +\infty$, as we wanted to prove.
\end{proof}

\begin{proof}[Proof of Theorem \ref{thm-calib.direct.horo}]
For $h>0$ and $a\in\Omega$,
let $u\in\calB_h(a)$ be a given horofunction directed by $a$.
This means that there is a sequence of configurations
$(p_n)_{n>0}$, such that $p_n=\lambda_na+o(\lambda_n)$
with $\lambda_n\to+\infty$ as $n\to\infty$, and such that
\[
u(x)=\lim_{n\to\infty}(u_{p_n}(x)-u_{p_n}(0))
\]
where $u_p$ denotes the function $u_p(x)=\phi_h(x,p)$.
Let also $\gamma:[0,+\infty)\to E^N$ be the curve
given by the hypothesis and satisfying
\[
u(\gamma(0))-u(\gamma(t))=\calA_{L+h}(\gamma\mid_{[0,t]})
\]
for all $t>0$.
In particular $\gamma$ is an $h$-minimizer.
We recall that this means that the restrictions of
$\gamma$ to compact intervals are global
minimizers of $\calA_{L+h}$.
Thus the restriction of $\gamma$ to $(0,+\infty)$ is
a genuine motion of the $N$-body problem,
with energy constant $h$,
and it is a maximal solution if and only if
$\gamma(0)$ has collisions,
otherwise the motion defined by $\gamma$
can be extended as a motion 
to some interval $(-\epsilon, +\infty)$.

The proof is divided into three steps.
The first one will be to prove that the curve $\gamma$
is not a superhyperbolic motion. This will be deduced from
the minimization property of $\gamma$.
Then we will apply the Marchal-Saari theorem to conclude
that there is a configuration $b\neq 0$ such that
$\gamma(t)=tb+O(t^{2/3})$.
The second and most sophisticated step will be to
exclude the possibility of having collisions in $b$,
that is to say, in the limit shape of the motion $\gamma$.
Finally, once it is known that $\gamma$ is a hyperbolic motion,
an easy application of the Chazy's Lemma
\ref{lema-cont.limitshape} will allow us to conclude
that we must have $b=\lambda a$ for some $\lambda>0$.
Then the proof will be achieved by observing that, since
$\norm{b}=\sqrt{2h}$,
we must also have $\lambda=\sqrt{2h}\norm{a}^{-1}$.

We start now by proving that $\gamma$ is not superhyperbolic.
We will give a proof by contradiction.
Supposing that $\gamma$ is superhyperbolic
we can choose $t_n\to +\infty$ such
that $R(t_n)/t_n\to +\infty$.
We recall that $R(t)=\max\set{ r_{ij}(t)\mid i<j}$
denotes the maximal distance between the bodies at time $t$,
and that $R(t)=O(\norm{\gamma(t)})$.
Thus we can assume that
$\norm{\gamma(t_n)-\gamma(0)}/t_n\to +\infty$.
Given that the calibrating property implies that the curve
$\gamma$ is an $h$-minimizer,
for each $n>0$ we have
\[
\calA_L(\gamma\mid_{[0,t_n]})+ht_n=
\phi_h(\gamma(0),\gamma(t_n)).
\]
Let us write for short $r_n=\norm{\gamma(0)-\gamma(t_n)}$.
In view of the observation we made in
Remark \ref{rmk-usual.bound},
and using Theorem \ref{thm-phih.estim},
 we have the lower and upper bounds
\[
\tfrac{1}{2}\,r_n^2\;t_n^{-1} +ht_n\;\leq\;
\phi_h(\gamma(0),\gamma(t_n))\;\leq\;
\left(\alpha\;r_n+h\beta\; r_n^2\right)^{1/2}
\]
for some constants $\alpha,\beta >0$ and for any $n>0$.
It is not difficult to see that this is impossible for $n$ large
enough using the fact that $r_n\,t_n^{-1}\to +\infty$.
Thus by the Marchal-Saari theorem there is a configuration
$b\in E^N$ such that $\gamma(t)=tb+O(t^{2/3})$.
Since by the Lagrange-Jacobi identity $b=0$ forces $h=0$,
we know that $b\neq 0$.

We prove now that $b$ has no collisions, that is to say,
that $b\in\Omega$. This is our second step in the proof.
Let us write $p=\gamma(0)$, $q_0=\gamma(1)$
and let us also define $\sigma_0\in\calC(q_0,p,1)$
by reversing the parametrization of
$\gamma_0=\gamma\mid_{[0,1]}$.
Thus $\sigma_0$ calibrates the function $u$,
that is to say, we have
$u(p)-u(q_0)=\calA_{L+h}(\sigma_0)$.

Now, using Lemma \ref{lema-JM.geod.complet} we can define
a sequence of curves $\sigma'_n\in\calC(p_n,q_0)$,
such that $\calA_{L+h}(\sigma'_n)=\phi_h(p_n,q_0)$ for all $n>0$.
Thus each curve $\sigma'_n$
is an $h$-calibrating curve of the function
$u_{p_n}(x)=\phi_h(x,p_n)$.
It will be convenient to also consider the curves $\gamma'_n$
obtained by reversing the parametrizations of
the curves $\sigma'_n$.
If for each $n>0$ the curve $\sigma'_n$ is defined over
an interval $[-s_n,0]$, then we get a sequence of curves
$\gamma'_n\in\calC(q_0,p_n,s_n)$,
respectively defined over the intervals $[0,s_n]$. 

Since $q_0$ is an interior point of $\gamma$,
Marchal's Theorem implies that $q_0\in\Omega$.
Thus for each curve $\gamma'_n$ the velocity
$w_n=\dot\gamma'_n(0)$ is well defined.
Since $h$-minimizers have energy constant $h$,
we also have $\norm{w_n}^2=2(h+U(q_0))$ for all $n>0$.
This allow us to choose a subsequence $n_k$ such that
$w_{n_k}\to v_0$ as $k\to\infty$.
At this point we need to prove
that $\lim s_n=+\infty$.
This can be done by application of Lemma \ref{lema-time.estim}
to the $h$-minimizers $\gamma'_n$ as follows.
Given two configurations $x,y\in E^N$,
the polynomial given by the lemma satisfies
$P(\tau)\geq \norm{x-y}^2-2\phi_h(x,y)\tau$ for all $\tau>0$.
Therefore,
when $x\neq y$ its roots can be bounded below by
$\norm{x-y}^2/2\phi_h(x,y)$.
Using this fact, we have that for all $n>0$,
\[
s_n>\frac{\norm{q_0-p_n}^2}{2\,\phi(q_0,p_n)}.
\]
Then the upper bound for $\phi_h$ given by
Theorem \ref{thm-phih.estim} implies that $\lim s_n=+\infty$. 

Let us summarize what we have built so far.
From now on, let us write for short $q_k=p_{n_k}$, $t_k=s_{n_k}$,
$v_k=w_{n_k}$, and also $\gamma_k=\gamma'_{n_k}$ and
$\sigma_k=\sigma'_{n_k}$.
First, there is a sequence of configurations $(q_k)_{k>0}$,
such that, for some increasing sequence $n_k$
of positive integers, we have
$q_k=\lambda_{n_k}a+o(\lambda_{n_k})$ as $k\to\infty$.
Associated to each $q_k$ there is an $h$-minimizer
$\gamma_k:[0,t_k]\to E^N$, with $t_k\to +\infty$,
such that $\gamma_k\in\calC(q_0,q_k)$.
Moreover, $v_k=\dot\gamma_k(0)$ and we have
$v_k\to v_0$ as $k\to\infty$.
In addition, each reversed curve $\sigma_k\in\calC(q_k,q_0)$
is an $h$-calibrating curve of the function
$u_{q_k}(x)=\phi_h(x,q_k)$.

We will prove that $v_0=\dot\gamma(1)$. 
To do this, we start by considering the maximal solution
of Newton's equations with initial conditions $(q_0,v_0)$
and by calling $\zeta$ its restriction to positive times,
let us say for $t\in [0,t^*)$.
Next, we choose $\tau\in (0,t^*)$ and we observe that
we have $t_k>\tau$ for any $k$ big enough.
Thus, for these values of $k$, we have that
$\gamma_k(t)$ and $\dot\gamma_k(t)$ converge respectively
to $\zeta(t)$ and $\dot\zeta(t)$, and the convergence is uniform
for $t\in[0,\tau]$.
Therefore,
\[
\lim_{k\to\infty}\calA_{L+h}(\gamma_k\mid_{[0,\tau]})=
\calA_{L+h}(\zeta\mid_{[0,\tau]}).
\]
On the other hand, since on each compact set our function
$u(x)$ is the uniform limit of the functions
$u_k(x)=u_{q_k}(x)-u_{q_k}(0)$, we can also write
\[
u(q_0)-u(\zeta(\tau))=
\lim_{k\to\infty}\,(\,u_k(q_0)-u_k(\gamma_k(\tau))\,).
\]
We use now the fact that for each one of these values of $k$
we have, by the calibration property, that
\[
u_k(q_0)-u_k(\gamma_k(\tau))=
\calA_{L+h}(\gamma_k\mid_{[0,\tau]}),
\]
to conclude then that
\[
u(q_0)-u(\zeta(\tau))=\calA_{L+h}(\zeta\mid_{[0,\tau]}).
\]
Notice that what we have proved is that the reversed curve
$\zeta(-t)$ defined on $[-\tau,0]$
is indeed an $h$-calibrating curve of $u$.
The concatenation of this calibrating curve with the
calibrating curve $\sigma_0$
results, according to Lemma \ref{lema-concat.calib},
in a new calibrating curve,
defined on $[-\tau,1]$ and passing by $q_0$ at $t=0$.
Therefore this concatenation of curves is an $h$-minimizer,
which implies that it is smooth at $t=0$.
We have proved that $\dot\zeta(0)=v_0=\dot\gamma(1)$.
This also implies that $t^*=+\infty$ and that
$\zeta(t)=\gamma(t+1)$ for all $t\geq 0$.

\begin{figure}[h]
\centering
\includegraphics[scale=1.1]{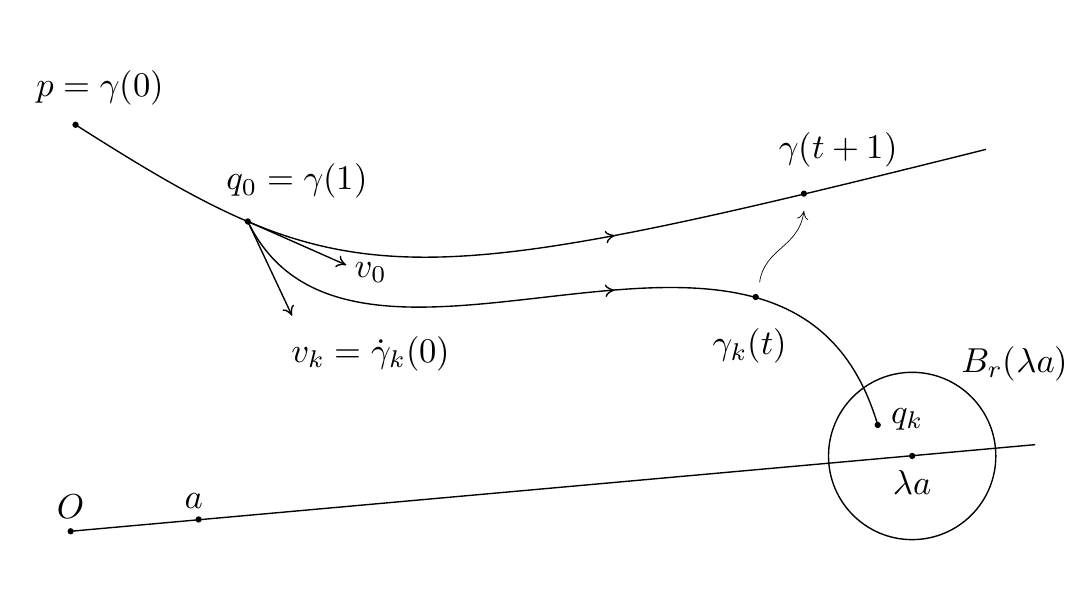}
\caption{The $C^1$ approximation of the curve $\gamma$ by
$h$-minimizers from $q_0$ to $q_k=p_{n_k}$.
Here $\lambda=\lambda_{n_k}$ and
$\norm{q_k-\lambda a}<r=o(\lambda)$.} 
\label{Proof33a}
\end{figure}

For simplicity,
in the rest of the proof we will call $\gamma$ the curve $\zeta$,
assuming then that the original curve $\gamma$ was
reparametrized to be defined on the interval $[-1,+\infty)$.
Making this abuse of notation we can then write
$\gamma_k(t)\to\gamma(t)$,
and $\dot\gamma_k(t)\to\dot\gamma(t)$,
uniformly on any compact interval $[0,T]$.

We continue now with the proof that the limit shape $b$
of $\gamma $ has no collisions.
We will make use of the function $\mu$ that we have mentioned
in Remark \ref{rmk-configur.measure} which is
called the configurational measure.
It is defined as the homogeneous function of degree zero
$\mu:E^N\setminus\set{0}\to\R^+$ given by
$\mu(x)=\norm{x}U(x)=U(\norm{x}^{-1}x)$,
that is to say, $\mu=UI^{1/2}$. 
Notice that $\mu(x)<+\infty$ if and only if $x\in\Omega$.

Under the assumption that $b$ has collisions,
we will construct a new sequence of curves
$\eta_k\in\calC(q_0,q_k)$ in such a way that
$\calA_{L+h}(\eta_k)<\calA_{L+h}(\gamma_k)$
for all $k$ big enough.
Since this contradicts the minimality of the curves $\gamma_k$
we will conclude that $b\in\Omega$.
The construction of the curves $\eta_k$ will be done in terms
of the polar components of the curves $\gamma_k$.
More precisely, for each $k>1$ we define the functions
\[
\rho_k:[0,t_k]\to\R^+,
\quad
\rho_k(t)=\norm{\gamma_k(t)}
\]
\[
\theta_k:[0,t_k]\to\S,
\quad
\theta_k(t)=\norm{\gamma_k(t)}^{-1}\gamma_k(t)
\]
where $\S=\set{x\in E^N\mid \inner{x}{x}=1}$ is the unit sphere
for the mass inner product. 
Thus, for each $k>0$ we can write $\gamma_k=\rho_k\theta_k$,
and the Lagrangian action in polar coordinates writes
\[
\calA_{L+h}(\gamma_k)=
\int_0^{t_k}\tfrac{1}{2}\;\dot\rho_k^{\,2}\,dt\,+
\int_0^{t_k}\tfrac{1}{2}\;\rho_k\,\dot\theta_k^{\,2}\,dt\,+
\int_0^{t_k}\rho_k^{-1}\,\mu(\gamma_k)\,dt\;+
ht_k.
\]
Assuming that $\mu(b)=+\infty$, we can find $\epsilon>0$
such that, if $\norm{x-b}<\epsilon$,
then $\mu(x)>3\mu(a)$.
On the other hand, since we have that $\gamma(t)=tb+o(t)$,
there is $T_0>0$ such that
$\norm{\gamma(t)t^{-1}-b}<\epsilon/2$ for all $t\geq T_0$.

We use now the approximation of $\gamma$ by the curves
$\gamma_k$. For each $T\geq T_0$ there is a positive integer
$k_T$ such that, if $k>k_T$,
then $t_k>T$ and
$\norm{\gamma_k(t)-\gamma(t)}<T_0\epsilon/2$
for all $t\in[T_0,T]$.
It follows that, for $k>k_T$ and for any $t\in[T_0,T]$
we have
\[
\norm{\frac{\gamma_k(t)}{t}-\frac{\gamma(t)}{t}}<
\frac{\epsilon}{2},
\]
and then $\norm{\gamma_k(t)t^{-1}-b}<\epsilon$.
In turn, since $\mu$ is homogeneous, this implies that
\[
\mu(\gamma_k(t))=
\mu(\gamma_k(t)t^{-1})>3\mu(a).
\]

Now we are almost able to define the sequence of curves
$\eta_k\in\calC(q_0,q_n)$.
Let us write $k_0$ for $k_{T_0}$.
For $k\geq k_0$ we know that
$\mu(\gamma_k(T_0))>3\mu(a)$.
Moreover, since the extreme $p_k$ of the
curve $\gamma_k$ lies in a ball $B_r(\lambda a)$ with
$r=o(\lambda)$, we can assume that $k_0$ is big enough
in order to have $\mu(p_k)<2\mu(a)$ for all $k\geq k_0$.
Then we define
\[
T_k=
\max\set{T\geq T_0\mid
\mu(\gamma_k(t))\geq 2\mu(a)
\text{ for all } t\in [T_0,T]},
\]
and $c_k=\theta_k(T_k)$.
Given $T>T_0$,
by the previous considerations we have that
$k>k_T$ implies $T_k>T$.
Thus,
we can take $T_k$ as large as we want
by choosing $k$ large enough.
The last ingredient for building the curve $\eta_k$
is a minimizer $\delta_k$ of $\calA_{L+h}$ in
$\calC(\gamma_k(T_0),\rho_k(T_0)c_k)$
whose existence is guaranteed by Theorem
\ref{lema-JM.geod.complet}.
Then we define $\eta_k$ as follows.
For $k<k_0$ we set $\eta_k=\gamma_k$.
For $k\geq k_0$ the curve $\eta_k$ is the concatenation
of the following four curves: (i) the restriction of $\gamma_k$
to $[0,T_0]$, (ii) the minimizer $\delta_k$ above defined,
(iii) the homothetic curve $\rho_k(t)c_k$ for $t\in [T_0,T_k]$,
and (iv) the restriction of $\gamma_k$ to $[T_k,t_k]$
(see Figure \ref{Proof33b}).

We will show that
$\Delta_k=\calA_{L+h}(\gamma_k)-\calA_{L+h}(\eta_k)>0$
for $k$ large enough.

We start by observing that the first and the last components
of $\eta_k$ are also segments of $\gamma_k$ so that
their contributions to $\Delta_k$ cancel each other out.

Also we have
\[
\calA_{L+h}(\gamma_k\mid_{[T_0,T_k]}) = 
\int_{T_0}^{T_k}\,\tfrac{1}{2}\,\dot\rho_k^{\,2}\,dt +
\int_{T_0}^{T_k}\,\tfrac{1}{2}\,\rho_k\dot\theta^{\,2}\,dt +
\int_{T_0}^{T_k}\,\rho_k^{-1}\mu(\gamma_k)\,dt +
h(T_k-T_0),
\]
and
\[
\calA_{L+h}(\,\rho_kc_k\mid_{[T_0,T_k]}) =
\int_{T_0}^{T_k}\,\tfrac{1}{2}\,\dot\rho_k^{\,2}\,dt +
\int_{T_0}^{T_k}\,\rho_k^{-1}\,2\mu(a)\,dt +
h(T_k-T_0).
\]
We recall that
$\mu(\gamma_k(t))\geq 2\mu(a)$ for all
$t\in [T_0,T_k]$.
Therefore, so far we can say that
\[
\Delta_k>
\int_{T_0}^{T_k}\rho_k^{-1}\,
\left(\mu(\gamma_k(t))-2\mu(a)\right)\,dt\;-\;
\calA_{L+h}(\delta_k).
\]
This part of the proof is essentially done.
To conclude we only need to establish estimates
for the two terms on the right side of the previous inequality.
More precisely,
we will prove that the the integral diverges as $k\to\infty$,
and that the second term is bounded as a function of $k$.

\begin{figure}[h]
\centering
\includegraphics[scale=1.1]{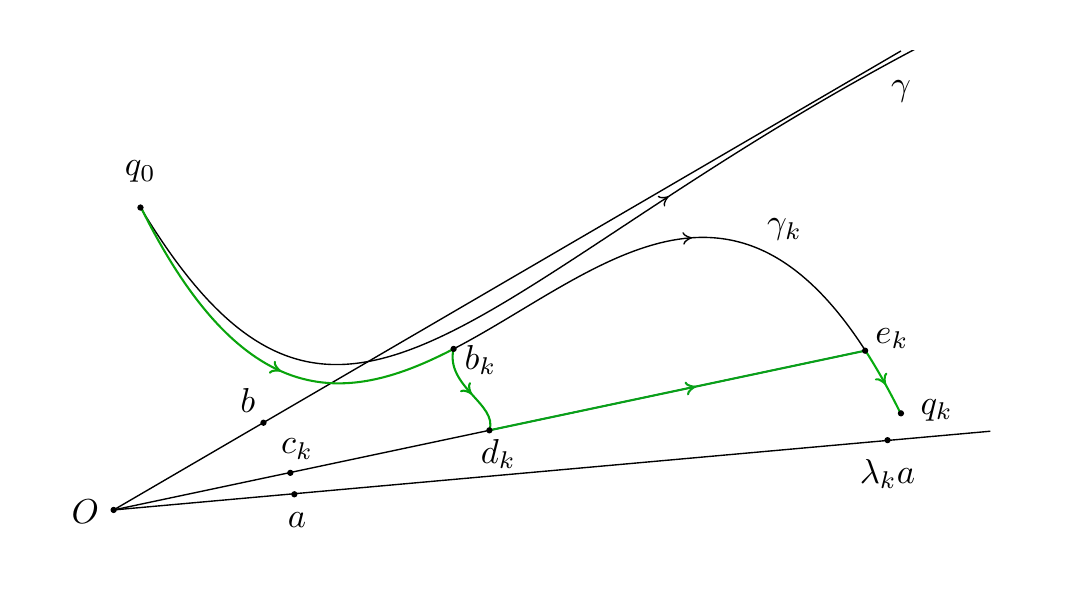}
\caption{For $k$ large enough, the $\calA_{L+h}$ action
of the green curve $\eta_k$
is less than that of the curve $\gamma_k$.
The intermediate points are 
$b_k=\gamma_k(T_0)$,
$d_k=\rho_k(T_0)c_k$,
and $e_k=\rho_k(T_k)c_k=\gamma_k(T_k)$.} 
\label{Proof33b}
\end{figure}

\begin{claim}
The sequence $\calA_{L+h}(\delta_k)$ is bounded.
\end{claim}
\begin{proof}
Indeed, the curve $\delta_k$ is a minimizer of $\calA_{L+h}$
between curves binding two configurations of size
$\rho_k(T_0)$, and
\[
\rho_k(T_0)\to\rho(T_0)=\norm{\gamma(T_0)}
\]
as $k\to\infty$. Therefore there is $R>0$ such that
the endpoints of the curves $\delta_k$ are all contained
in the compact ball $B_R(0)\subset E^N$.
On the other hand, 
since by Theorem \ref{thm-phih.estim} we know that
the action potential $\phi_h$ is continuous, we can conclude that
$\sup \calA_{L+h}(\delta_k)<+\infty$.
\end{proof}

\begin{claim}
The sequence
$\int_{T_0}^{T_k}\rho_k^{-1}\,
\left(\mu(\gamma_k(t))-2\mu(a)\right)\,dt$
diverges as $k\to\infty$.
\end{claim}
\begin{proof}
In order to get a lower bound for the integral of $\rho_k^{-1}$,
we make the following considerations.
We note first that $\rho(t)=\norm{\gamma(t)}< \alpha t+\beta$
for some constants $\alpha, \beta>0$.
This is because we know that
$\gamma(t)=tb+o(t)$ as $t\to +\infty$.
Thus we have that for any $T>T_0$
\[
\int_{T_0}^T\rho^{-1}dt \;\geq\;
\log(\alpha T+\beta)-\log(\alpha T_0 +\beta).
\]
Therefore, for any choice of $K>0$ there is $T>0$ such that
the integral at the left side is bigger than
$\mu(a)^{-1}K$.

On the other hand,
since for $k>k_T$ we have that $T_k>T$, and since
$\gamma_k(t)$ uniformly converges to $\gamma(t)$
on $[T_0,T]$, we can assume that we have
$\mu(\gamma_k(t))>3\mu(a)$ for all $t\in [T_0,T]$
and then,
neglecting the part of the integral between $T$ and $T_k$
which is positive, to conclude that
\[
\int_{T_0}^{T_k}\rho_k^{-1}dt\,
\left(\mu(\gamma_k(t))-2\mu(a)\right)\,dt \;>
\;\mu(a)\int_{T_0}^{T}\rho_k^{-1}dt \;>\;K
\]
for every $k$ sufficiently large.
\end{proof}
It follows that for large values of $k$ the difference $\Delta_k$ is
positive, meaning that the corresponding curves $\gamma_k$
are not $h$-minimizers because the curves $\eta_k$ have
smaller action.
Therefore we have proved by contradiction that $b\in\Omega$.

The last step to finish the proof is to show that $b=\lambda a$
for some $\lambda >0$.
If not, we can choose two disjoint cones $C_a$ and $C_b$
in $E^N$, centered at the origin and with axes directed by
the configurations $a$ and $b$ respectively.
Since we know that $b\in\Omega$,
we can apply Chazy's Lemma to get that for $k$ large enough
the curves $\gamma_k$ are defined for all $t>0$, and that there
is $T^*>0$ for which we must have $\gamma_k(t)\in C_b$
for all $t>T^*$ and any $k$ large enough.
But this produces a contradiction, because we know that
$q_k=\gamma_k(t_k)=\lambda_{n_k}a+o(\lambda_{n_k})$
as $k\to\infty$, which
forces to have $q_k\in C_a$ for $k$ large enough.
\end{proof}

\section{The Jacobi-Maupertuis distance for nonnegative energy}
\label{s-jm.dist}

In this section we develop the geometric viewpoint
and we show, for $h\geq 0$, that when restricted to $\Omega$
the action potential $\phi_h$ is exactly the Riemannian distance
associated to the Jacobi-Maupertuis metric $j_h=2(h+U)g_m$,
where $g_m$ is the mass scalar product.
Moreover,
we will see that the metric space $(E^N,\phi_h)$
is the completion of $(\Omega,j_h)$.
The fact that $\phi_h$ is a distance over $E^N$ is a
straightforward consequence of the definition and of
Lemma \ref{lema-geod.are.lipschitz} or
Lemma \ref{lema-compar.pot.action.kepler}
depending on whether $h>0$ or $h=0$.  
It is also immediate to see that $(E^N,\phi_h)$ is a length space,
that is to say $\phi_h$ coincides with the induced length distance.
From now on,
we denote by $\calL_h(\gamma)$ the Riemannian length
of a $C^1$ curve $\gamma$,
and we denote by $d_h$ the Riemannian distance on $\Omega$.
 
\begin{proposition}
For all $h\geq 0$,
the space $(E^N,\phi_h)$ is the completion of $(\Omega,d_h)$.    
\end{proposition}

\begin{proof}
In the case $h>0$,
the fact that $(E^N,\phi_h)$ is a complete length space
comes directly from the definition of $\phi_h$
and from Lemma \ref{lema-geod.are.lipschitz}
and Theorem \ref{thm-phih.estim}.
Moreover, we have that $\phi_h$ generates the topology of $E^N$
and that $\Omega$ is thus a dense subset.

For the case $h=0$ the argument is exactly the same,
but instead of Lemma \ref{lema-geod.are.lipschitz},
which becomes meaningless,
we have to use Lemma \ref{lema-compar.pot.action.kepler} below.

The proof will be achieved now by showing that
the inclusion of $\Omega$ into  $E^N$ is an isometry,
that is to say,
that $\phi_h$ coincides with $d_h$ when restricted to $\Omega$.
Given $(x,v)\in \Omega\times E^N$,
we have
\[
\norm{v}_h=j_h(x)(v,v)^{1/2}\leq L(x,v)+h
\]
with equality if and only if $\calE(x,v)=h$,
where $\calE(x,v)=\frac{1}{2}\norm{v}^2-U(x)$
is the energy function in $T\Omega$.
It follows that if $\gamma$ is an absolutely continuous curve in
$\Omega$ it holds $\calL_h(\gamma)\leq A_{L+h}(\gamma)$,
with equality if and only if
$\calE(\gamma(t),\dot{\gamma}(t))=h$  for almost all $t$.
Given now $x,y\in\Omega$,
by Marchal's Theorem any $h$-minimizer joining $x$ to $y$
is a genuine motion,
in particular it is a $C^1$ curve.
Since $d_h$ is defined as the infimum of $\calL_h(\gamma)$
over all $\calC^1$ curves in $\Omega$ joining $x$ to $y$,
we have that $d_h(x,y)\leq\phi_h(x,y)$.

In order to prove the converse inequality,
let $\epsilon>0$ and $\gamma:[0,1]\to \Omega$
be a $\calC^1$ curve joining $x$ to $y$ such that
$\calL_h(\gamma)\leq d_h(x,y)+\epsilon$.
We can now find a finite sequence $0=t_0<...<t_N=1$
such that for any $i=1,...,N$ the points
$\gamma(t_{i-1})$ and $\gamma(t_i)$ can be joined by a
minimizing geodesic in $\Omega$, here denoted $\sigma_i$.
We will assume that each $\sigma_i$ is parametrized by arclength,
thus $\dot{\sigma}_i(t)\neq 0$ for all $t$.
Let us reparametrize now each $\sigma_i$ so that,
denoting $\delta_i$ the reparametrization,
we have $\calE(\delta_i(t),\dot\delta_i(t))=h$ for all $t$,
and let $\delta$ be the concatenation of all $\delta_i$. 
By construction
\[
\phi_h(x,y)\leq A_{L+h}(\delta)=
\calL_h(\delta)\leq
\calL_h(\gamma)\leq
d_h(x,y)+\epsilon,
\]  
and by arbitrariness of $\epsilon$ we conclude that
$\phi_h(x,y)\leq d_h(x,y)$.
\end{proof}  

\begin{lemma}
\label{lema-compar.pot.action.kepler}
There exist a constant $\mu_0>0$ such that
for all $x,y\in E^N$ satisfying $x\neq y$, we have
\[
\phi_0(x,y)\geq \frac{\mu_0}{\rho}\norm{x-y}
\]
where
$\rho=\max\set{\norm{x},\norm{y}}^\frac{1}{2}$.
\end{lemma}

\begin{proof}
The main idea of the proof is to estimate $\phi_0$
by comparing it with the action of some Kepler problem in $E^N$.
Since $U$ is a continuous function with values in $(0,+\infty]$,
the minimum of $U$ on the unit sphere of $E^N$,
here denoted $U_0$, is strictly positive.
Thus, by homogeneity of the potential,
if $x$ is any nonzero configuration we have
\[
U(x)=\frac{1}{\norm{x}}\,U\left(\frac{x}{\norm{x}}\right)
\geq \frac{U_0}{\norm{x}}.
\]
Let us consider now the Lagrangian function
associated to the Kepler problem in $E^N$
with potential $U_0/\norm{x}$, that is to say 
\[
L_{\kappa}(x,v)=
\frac{1}{2}\norm{v}^2+\frac{U_0}{\norm{x}}\;.
\]
By the previous inequality we know that $L_\kappa(x,v) \le L(x,v)$.
The critical action potential associated to $L_\kappa$
is defined on $E^N\times E^N$ by
\[
\Phi_0(x,y)=
\min\set{\calA_{L_\kappa}(\gamma)\mid \gamma\in\calC(x,y)},
\]
and it follows immediately from the definition that
$\Phi_0(x,y)\le \phi_0(x,y)$.
Assume now $x\neq y$, and let
$\gamma : [0,\tau]\rightarrow E^N$ be a free-time minimizer for
$\calA_{L_\kappa}$ in $\calC(x,y)$.
Thus $\gamma$ is an absolutely continuous curve satisfying
$\calA_{L_\kappa}(\gamma)=\Phi_0(x,y)$.
As a zero energy motion of the Kepler problem,
we know that $\gamma$ is an arc of Keplerian parabola,
and in particular we know that
\[
\max_{t\in[0,\tau]}\norm{\gamma(t)}=
\max\set{\norm{x},\norm{y}}
\]
which in turn implies that
\[
\frac{U_0}{\norm{\gamma(t)}}\,\geq\,
\frac{U_0}{\rho^2}
\]
for all $t\in[0,\tau]$. 
Thus, using this lower bound and Cauchy-Schwarz inequality
for the kinetic part of the action of $\gamma$ we deduce that 
$\Phi_0(x,y)\geq g(\tau)$,
where $g:\R^+\to\R$ is the function defined by 
\[
g(s)=
\frac{\norm{x-y}^2}{2s}+\frac{U_0}{\rho^2}\,s.
\]
Observing now that $g$ is convex and proper,
and replacing $g(\tau)$ in the previous inequality
by the minimum of $g(s)$ for $s>0$, we obtain
\[
\phi_0(x,y)\geq
\Phi_0(x,y)\geq
\frac{\mu_0}{\rho}\,\norm{x-y}.
\]
for $\mu_0=\sqrt{2U_0}$.
\end{proof} 

Now we have all the necessary elements to give
the proof of the corollary stated in Sect. \ref{s-geom.view}.
We have to prove that if two geodesic rays have the same
asymptotic limit, then they are equivalent in the sense of
having bounded difference.

\begin{proof}
[Proof of Corollary \ref{cor-Gr.boundary}]
Let $\gamma:[0,+\infty)\to E^N$ be a geodesic ray 
of the distance $\phi_h$, with $h>0$.
We assume that $\gamma(s)=sa+o(s)$
as $s\to +\infty$ for some $a\in\Omega$.
Thus, we know that $\gamma(s)$ is without collisions
for all $s$ sufficiently big. 
By performing a time translation we can assume that
$\gamma(s)\in \Omega$ for all $s\geq 0$, hence that
$\gamma$ is a geodesic ray of the
Jacobi-Maupertuis metric $j_h$ in $\Omega$.
Now we know that $\gamma$ admits a factorization
$\gamma(s)=x(t_\gamma(s))$ where
$x(t)$ is a motion of energy $h$.
More precisely,
the inverse of the new parameter $t_\gamma$
is a function $s_x$ satisfying $x(t)=\gamma(s_x(t))$.
Since $\gamma$ is arclength parametrized,
we have $\norm{\dot\gamma(s)}_h=1$ for all $s\geq 0$,
and we deduce that $s_x$ is the solution
of the differential equation 
\begin{equation}\tag{$\star$}\label{eq-edo.reparam}
\dot s_x(t)=2h+2U(\gamma(s_x(t)))
\end{equation}
with intial condition $s_x(0)=0$.
This implies that $s_x(t)\to +\infty$ and $\dot s_x(t)\to 2h$
as $t\to +\infty$,
hence we also have $s_x(t)=2ht+o(t)$
and $x(t)=2ht\,a+o(t)$ as $t\to +\infty$.
In particular $x(t)$ is a hyperbolic motion.
We claim now that
\[
s_x(t)=2ht+\frac{U(a)}{h}\log t+O(1).
\]
and the proof is as follows.
From (\ref{eq-edo.reparam}) we have, for $t>1$,
\begin{equation}\tag{$\star\star$}\label{eq-sx.asymptotic}
s_x(t)=2ht+\int_0^1 2U(x(\nu))\,d\nu+\int_1^t 2U(x(\nu))\,d\nu.
\end{equation}
On the other hand, by Chazy's Theorem we have that
\[
x(t)=
2ht\,a-\frac{\log t}{4h^2}\,\nabla U(a)+O(1).
\]
We observe then that
\begin{eqnarray*}
U(x(\nu))&=&
\frac{1}{2h\,\nu}\,U\left(a +O\left(\frac{\log \nu}{\nu}\right)\right)\\
& &\\
&=&\frac{U(a)}{2h}\,\frac{1}{\nu}+
O\left(\frac{\log \nu}{\nu^2}\right)
\end{eqnarray*}
Now the claim can be verified by replacing this last expression of
$U(x(\nu))$ in the last term of (\ref{eq-sx.asymptotic}).

Given now another gedodesic ray
$\sigma : [0,+\infty)\to E^N$,
denoting $\sigma(s)=y(t_\sigma(s))$
the reparametrization such that $y(t)$ is a motion
of energy constant $h$,
and denoting $s_y$ the inverse of $t_\sigma$,
it is clear from the previous asymptotic estimates that the
difference $s_x(t)-s_y(t)$ is bounded.
Since the derivative of $s_x$ and $s_y$
are both bounded below by the same positive constant,
we easily conclude that
$t_\gamma(s)-t_\sigma(s)$ is also bounded.
By replacing in the asymptotic expansion of $x(t)$ and $y(t)$
we find that $\gamma(s)-\sigma(s)$ is bounded. 
\end{proof}

\section{Open questions on bi-hyperbolic motions}
\label{s-bi.hyp}

We finish with some general open questions.
They are closely related to the recent advances made by
Duignan {\it et al.} \cite{DuMoMoYu}
in which the authors show in particular that
the limit shape map $(x,v)\mapsto (a^-,a^+)$
defined below is actually real analytic.

We define bi-hyperbolic motions
as those which are defined for all $t\in\R$,
and are hyperbolic both in the past and in the future.
The orbits of these entire solutions
define a non-empty open set in the phase space,
namely the intersection of the two open set
\[\calH=\calH^ +\cap\calH^-\]
where
$\calH^+\subset T\Omega=\Omega\times E^N$
is the set of the initial conditions giving rise
to hyperbolic motions in the future, and
$\calH^-=\set{(x,v)\in T\Omega\mid (x,-v)\in\calH^+}$
is the set of the initial conditions
giving rise to hyperbolic motions in the past.
Newton's equations define a complete vector field
in the open set $\calH\subset\Omega\times E^N$.
We will denote by $\varphi^t$ the corresponding flow
and $\pi:\Omega\times E^N\to\Omega$
the projection onto the first factor.
 
We also note that this open and completely invariant set
has a natural global section,
given by the section of \emph{perihelia}:
\[\calP=\calH\cap\set{(x,v)\in T\Omega\mid \inner{x}{v}=0}\,.\]

\begin{proposition}
The flow $\varphi^t$ in $\calH$
is conjugated to the shift in $\calP\times\R$.
\end{proposition}

\begin{proof}
Given $(x_0,v_0)\in\calH$,
let $x(t)=\pi(\varphi^t(x_0,v_0))$ be
the generated bi-hyperbolic motion.
Since $I=\inner{x}{x}$,
it follows from the Lagrange-Jacobi identity $\ddot I=4h+2U$,
that $I$ is a proper and strictly convex function.
Thus, there is a unique $t_p\in\R$ such that 
$\varphi^{t_p}(x_0,v_0)\in\calP$.
Moreover,
the sign of $\dot I=\inner{x}{\dot x}$ is the sign of $t-t_p$
and $\norm{x(t)}$ reaches its minimal value at $t=t_p$.
The conjugacy is thus given by the map
$(x_0,v_0)\mapsto (p(x_0,v_0),-t_p)$,
where $p:\calH\to\calP$ gives the phase point at perihelion
$p(x_0,v_0)=(x(t_p),\dot x(t_p))$.
\end{proof}

Naturally associated with each bi-hyperbolic motion,
there is the pair of limit shapes that it produces
both in the past and in the future.
More precisely, we can define the \emph{limit shape map}
$S:\calH\to\Omega\times\Omega$ by
\[S(x,v)=(a^-(x,v),a^+(x,v))\]
\[a^\pm(x,v)=\lim_{t\to\pm\infty} \;\abs{t}^{-1}\pi(\varphi^t(x,v))\,.\]
As a consequence of Chazy's \emph{continuity of the instability}
(Lemma \ref{lema-cont.limitshape}) we have that
the limit shape map is actually a continuous map.
It is also clear that
\[\norm{a^-(x,v)}=\norm{a^+(x,v)}\] for all $(x,v)\in\calH$.
In fact, we have
\[\norm{a^\pm(x,v)}^2=2h=\norm{v}^2-2U(x)\]
where $h>0$ is the energy constant
of the generated bi-hyperbolic motion.
Hence the image of $S$ is contained in the manifold
\[
\calS=
\set{(a,b)\in\Omega\times\Omega\mid \norm{a}=
\norm{b}}\,.
\]
Clearly, we have $S\circ\varphi^t=S$ for all $t\in\R$.
Therefore the study of the limit shape map can be restricted
to the section of perihelia $\calP$.
Counting dimensions we get
\[\dim\calP=2dN-1=\dim\calS\]
where $d=\dim E$.

We will see now that the center of mass
can be reduced to the origin.
Let us call $G:E^N\to E$ the linear map that associates
to each configuration its center of mass.
More precisely,
if $M=m_1+\dots+m_N$ is the total mass of the system,
then the center of mass $G(x)$ of $x=(r_1,\dots,r_N)\in E^ N$
is well defined by the condition
$MG(x)=m_1r_1+\dots+m_Nr_N$.
Just as we did for the quantities U and I,
we will write $G(t)$ instead of $G(x(t))$ when
the motion $x(t)$ is understood.
We observe now that if $x(t)=ta^++o(t)$ as $t\to+\infty$,
then $G(t)=tG(a^+)+o(t)$.
Moreover, since $\ddot G(t)=0$ for all $t\in\R$ we know that
the velocity of the center of mass $\dot G(t)=v_G$ is constant,
hence $G(t)=tv_G+G(0)$.
Therefore we must have $G(a^+)=v_G$.
If in addition $x(t)=-ta^-+o(t)$ as $t\to-\infty$,
then we also have $G(a^-)=-v_G$.
We conclude that
\[G(a^-(x,v))=-\;G(a^+(x,v))\]
for all $(x,v)\in\calH$.
This allows to reduce in $d$ dimensions
the codomain of the limit shape map.
On the other hand,
a constant translation of a bi-hyperbolic motion
gives a new bi-hyperbolic motion with the same limit shapes.
Thus the domain can also be reduced of $d$ dimensions
by imposing the condition $G(x(0))=0$.

Finally,
we note that bi-hyperbolic motions are preserved
by addition of uniform translations.
Let $\Delta\subset E^N$ be the diagonal subspace,
that is the set of configurations of total collision. 
For any bi-hyperbolic motion $x(t)$
with limit shapes $a^-$ and $a^+$, 
and any $v\in\Delta$,
we get a new bi-hyperbolic motion $x_v(t)=x(t)+tv$,
whose limit shapes are precisely $a^--v$ and $a^++v$.
In particular these configurations without collisions
have opposite center of mass and the same norm.
The equality of the norms can also be deduced
from the orthogonal decomposition
$E^N=\Delta\oplus \ker G$ and using the fact that
$G(a^+-a^-)=0$.

In sum,
we can perform the total reduction of the center of mass
by setting $G(x(0))=G(\dot x(0))$,
which leads to $G(a^-)=G(a^+)=0$.
We define
\[\calP_0=\set{(x,v)\in\calH\mid\;
G(x)=G(v)=0\text{ and }\inner{x}{v}=0}\]
\[\calS_0=\set{(a,b)\in\Omega\times\Omega\mid\;
G(a)=G(b)=0\text{ and }\norm{a}=\norm{b}}\]
and we maintain the balance of dimensions.

\begin{question}
Is the limit shape map $S:\calP_0\to\calS_0$
a local diffeomorphism?
\end{question}

The answer is yes in the Kepler case
(see Figure \ref{HyperKepler1}).
But in the general case,
this property must depend on the potential $U$.
For instance,
in the extremal case of $U=0$,
in which motions are thus straight lines,
we get the restriction $a^-=-a^+$ for all hyperbolic motion.
In this case the shape map loses half of the dimensions.

It is therefore natural to ask, for the general $N$-body problem,
whether or not
there is some relationship between these two functions.

\begin{question}
How big is the image of the limit shape map?
\end{question}

In the Kepler case, only the pairs $(a,b)$ such that
$\norm{a}=\norm{b}$ and $a\neq\pm\, b$
are realized as asymptotic velocities of some hyperbolic trajectory.
This can be generalized for $N\geq 3$.
If $a\in\Omega$ is a planar central configuration
and $R\in SO(E)$ keeps invariant the plane containing $a$,
the pair $(a,Ra)$ is realized as the limit shapes
of a unique homographic hyperbolic motion,
except in the cases $R=\pm\,Id$.

\begin{figure}[h]
\centering
\includegraphics[scale=1.1]{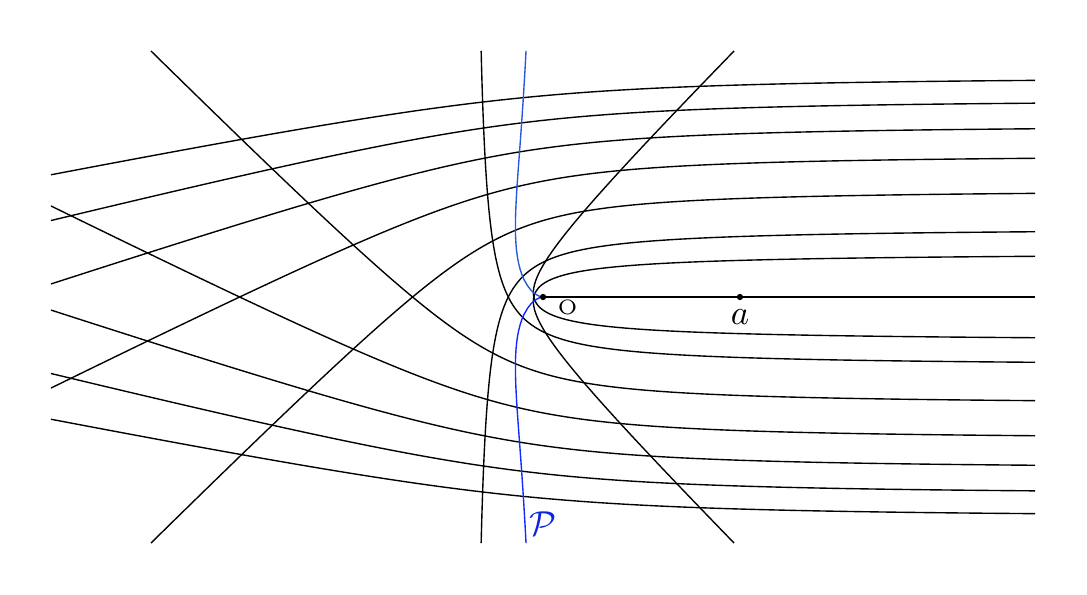}
\caption{Hyperbolic motions of the Kepler problem
with fixed value of the energy constant $h>0$
and asymptotic velocity $a$ in the future.
All but one of these motions are bi-hyperbolic.
The blue curve $\calP$ is composed of the
corresponding perihelia.} 
\label{HyperKepler1}
\end{figure}

We now devote attention to the effect of homogeneity.
Recall that
if $x(t)$ is a bi-hyperbolic motion of energy constant $h$,
then  for every $\lambda>0$ the solution given by
$x_\lambda(t)=\lambda\,x(\lambda^{-3/2}t)$
is still bi-hyperbolic with energy constant $\lambda^{-1}h$.
Moreover,
if we note $x_0=x(0)$ and $v_0=\dot x(0)$ then we have
\[
(x_\lambda(t),\dot x_\lambda(t))=
\varphi^t(\lambda\, x_0,\lambda^{-1/2}\,v_0)
\]
for all $t\in\R$.
These considerations prove the following remark.

\begin{remark}
For any $(x,v)\in\calH$ and for any $\lambda>0$, we have
\[S(\lambda\, x,\lambda^{-1/2}\,v)=\lambda^{-1/2}\,S(x,v)\,.\]
\end{remark}

Let us introduce the following question with an example.
Consider the planar three-body problem with equal masses.
That is, $E=\R^2\simeq\C$, $N=3$ and $m_i=1$ for $i=1,2,3$.
For $h>0$, define the equilateral and collinear configurations
\[a_h=\sqrt\frac{2h}{3}\,(1,z,z^2)\qquad b_h=\sqrt h\,(-1,0,1)\]
where $z$ is a primitive root of $z^3-1$.
Thus we have $\norm{a_h}=\norm{b_h}=\sqrt{2h}$ and also
$G(a_h)=G(b_h)=0$ for all $h>0$.

\begin{question}
Is the pair $(a_h,b_h)$ in the image of the limit shape map?
\end{question}

In other words,
is there a bi-hyperbolic motion whose dynamics originates
in the past with a contraction from a big equilateral triangle,
and then, after a period of strong interaction between the particles,
the evolution ends with an almost collinear expansion?

In our view, the method of viscosity solutions
could be useful to answer this question. 
In particular, we consider it necessary to push forward
the understanding of the regions of differentiability
of these weak solutions.
It seems reasonable that an orbit like this
can be found by looking for critical points
of a sum of two Busemann functions (see Sect. \ref{s-busemann}).

\begin{question}
If the answer to Question 3 is yes,
what is the infimum of the norm of the perihelia
of the bi-hyperbolic motions having these limit shapes? 
\end{question}

Observe that once we have a bi-hyperbolic motion which is
equilateral in the past and collinear in the future,
we can play with the homogeneity in order to obtain a new one,
but having a perihelion contained in an arbitrarily small ball.
That is to say, it would be possible to make, at some point,
all bodies pass as close as we want from a total collision.
Of course, to do this
we must increase the value of the energy constant indefinitely.
Thus we preserve the limit shapes in the weak sense,
but not the size of the asymptotic velocities.
 In the family of motions $(x_\lambda)$ described above,
the product of the energy constant $h$
and the norm of the perihelion is constant.
In the Kepler case,
once we fix the value of $h>0$
there is only one bi-hyperbolic motion
connecting a given pair $(a,b)$ (see \cite{Alb}).
Therefore we can see the norm of the perihelion
as a function of the limit shapes.
We can see that the norm of the perihelion
tends to $0$ for $a\to b$,
and tends to $+\infty$ for $a\to -b$.


\subsection*{Acknowledgements}
The authors would like to express
a very special thank to Albert Fathi,
who suggested the use of the method
of viscosity solutions, as well as the way
to construct the maximal calibrating curves
of the horofunctions. 
The first author is also grateful to
Alain Albouy, Alain Chenciner and Andreas Knauf
for several helpful conversations at the IMCCE.
Finally, we would like to thank the referee
for his accurate and helpful comments,
and Juan Manuel Burgos
for pointing out a subtle mistake in one of the proofs.


\end{document}